\documentclass[12pt]{article}
\usepackage[utf8]{inputenc}
\usepackage[english]{babel}

\usepackage{amsthm}
\usepackage{amsmath, amssymb, amsfonts}
\usepackage{mathrsfs}
\usepackage{titling}
\usepackage[colorlinks=true, linkcolor=blue, citecolor=green, urlcolor=cyan]{hyperref}
\usepackage{url}
\usepackage{bbm}
\usepackage{xcolor}
\usepackage{todonotes}
\usepackage{enumitem}
\usepackage{soul}
\usepackage{mathtools}

\usepackage[margin=1.2in]{geometry}

\newcommand{\A}{\mathcal{A}}
\newcommand{\B}{\mathsf{B}}
\newcommand{\C}{\mathcal{C}}
\newcommand{\CC}{\mathbf{C}}
\newcommand{\D}{\mathcal{D}}

\newcommand{\E}{\mathbb{E}}
\newcommand{\EE}{\mathbf{E}}
\newcommand{\F}{\mathcal{F}}
\newcommand{\I}{\mathcal{I}}
\newcommand{\II}{\mathbf{I}}
\newcommand{\N}{\mathbb{N}}
\renewcommand{\P}{\mathbb{P}}
\newcommand{\PP}{\mathbf{P}}
\newcommand{\Prm}{\mathrm{P}}

\newcommand{\Psf}{\mathsf{P}}
\newcommand{\Q}{\mathbf{Q}}
\newcommand{\R}{\mathbb{R}}
\newcommand{\T}{\mathbb{T}}
\newcommand{\Z}{\mathbb{Z}}

\newcommand{\1}{\mathbbm{1}}

\newcommand{\ajar}{\mathrm{ajar}}
\newcommand{\Ber}{\mathrm{Ber}}
\newcommand{\dist}{\mathrm{dist}}
\newcommand{\core}{\mathrm{core}}
\newcommand{\env}{\mathrm{env}}

\newcommand{\mix}{\mathrm{mix}}
\newcommand{\opt}{\!\mathrm{opt}}
\newcommand{\site}{\mathrm{site}}
\newcommand{\sta}{\mathrm{sta}}
\newcommand{\TV}{\mathrm{TV}}
\newcommand{\upd}{\mathrm{upd}}

\newtheorem{theorem}{Theorem}[section]
\newtheorem{lemma}[theorem]{Lemma}

\newtheorem{proposition}[theorem]{Proposition}
\theoremstyle{definition}
\newtheorem{definition}[theorem]{Definition}
\newtheorem{remark}[theorem]{Remark}
\newtheorem{claim}[theorem]{Claim}

\numberwithin{equation}{section}

\title{Mixing times of spin systems on dynamical percolation}
\author{
    Alexandre Stauffer\thanks{Department of Mathematics, King's College London. Email: \texttt{a.stauffer@kcl.ac.uk}.}
    \and 
    Oskar Vavtar\thanks{Department of Mathematics, King's College London. Email: \texttt{oskar.vavtar@kcl.ac.uk}.}
}
\date{\today}

\begin{document}
\maketitle


\begin{abstract}
    We study the mixing times of stochastic spin systems corresponding to nearest-neighbour Glauber dynamics on dynamical percolation, defined on $d$-dimensional torus of side-length $N$. In this model, the status of each edge (open or closed) updates independently at rate $\lambda>0$, according to $\Ber(p)$ samples. Simultaneously, the spin of each site updates at rate $1$ according to Glauber dynamics on the environment restricted to open edges. We show that for a relatively general class of nearest-neighbour systems, as long as $p<p_c(d)$, for any temperature, if $\lambda$ is sufficiently small, the mixing time is of order $\frac{\log N}{\lambda}$. This Markov chain is non-reversible, and the proof is obtained by developing a particular coupling that couples together local configurations whenever the environment behaves well.
\end{abstract}


\section{Introduction}\label{sec:intro}

Let $G=(V,E)$ be a finite graph and $S$ a finite set. We study the dynamics of a class of $S^V$-valued spin systems in a random dynamic environment. The latter will be given by dynamical percolation, which is a $\{0,1\}^E$-valued Markov chain, which represents opening and closing of edges in $E$. 

Fixing $p\in(0,1)$ and $\lambda>0$, we define the \textit{dynamical percolation with speed $\lambda$ and density $p$}, which we denote by $(\eta_t)_{t\geq 0}$, as follows:
\begin{itemize}
    \item[(i)] we start from an arbitrary (possibly random) initial configuration $\eta_0\in\{0,1\}^E$;
    \item[(ii)] each edge carries an independent rate-$\lambda$ Poisson clock;
    \item[(iii)] if the clock corresponding to edge $e$ rings at time $t$, we resample the value of $\eta_t(e)$ according to $\Ber(p)$, independently of anything else.  
\end{itemize}
By convention, we say that the edge $e$ is \textit{open} (resp.~\textit{closed}) at time $t$ if $\eta_t(e)=1$ (resp.~$\eta_t(e)=0$). It is easy to check that $(\eta_t)_{t\geq 0}$ is invariant with respect to Bernoulli (bond) percolation on $G$, i.e., the measure $\P_p=\P_p^G$, given by
$$\P_p(\eta) ~=~ \prod_{e\in E}p^{\eta(e)}(1-p)^{1-\eta(e)}, \quad \eta\in\{0,1\}^E.$$ 

Consider now a collection of potentials $\Phi=\{\Phi_{e}:e\in E\}$, where for each $e\in E$, the map $\Phi_e:S\times S\to\R$ is symmetric. That is, for $e=xy$ and $\sigma\in S^V$, we have
$$\Phi_{xy}(\sigma(x),\sigma(y)) ~=~ \Phi_{xy}(\sigma(y),\sigma(x)).$$
We define a spin system given by potentials $\Phi$ on dynamical percolation as a stochastic process $(\sigma_t,\eta_t)_{t\geq 0}$, which evolves as follows:
\begin{itemize}
    \item[(i)] $(\eta_t)_{t\geq 0}$ is taken to be dynamical percolation on $\{0,1\}^E$ with speed $\lambda$ and density $p$.
    \item[(ii)] $(\sigma_t)_{t\geq 0}$ is the $S^V$-valued spin component, started from arbitrary $\sigma_0\in S^V$, which evolves in the following way:
    \begin{itemize}
        \item[(a)] each site carries an independent rate-$1$ Poisson clock;
        \item[(b)] if the clock corresponding to site $x$ rings at time $t$, the value of $\sigma_t(x)$ resamples according to 
        \begin{equation}\label{eq:update_conditional_probability}
        \mu_{\sigma_t,\eta_t}^x(\cdot) ~:=~ \frac{\exp\!\big(\!-\sum_{y:xy\in E}\eta_t(xy)\Phi_{xy}(\cdot,\sigma_t(y))\big)}{\sum_{s\in S}\exp\!\big(\!-\sum_{y:xy\in E}\eta_t(xy)\Phi_{xy}(s,\sigma_t(y))\big)}.
        \end{equation}
    \end{itemize}
\end{itemize}
We will refer to this process as \textit{$\Phi$-spin system on dynamical percolation}. The process $(\sigma_t,\eta_t)_{t\geq 0}$ is also a Markov chain; assuming that $\Phi$ is such that the process is irreducible, we will write $\pi$ for the corresponding invariant measure. Note, however, that $(\sigma_t)_{t\geq 0}$ by itself is not a Markov chain, as the transition rates depend on the environment. Moreover, the chain $(\sigma_t,\eta_t)_{t\geq 0}$ is not reversible, in the sense that there is no measure on $S^V\times\{0,1\}^E$ with respect to which the chain is reversible. 

We note that if in the part (i) of the definition above, we did not consider dynamical percolation but simply took $\eta_t\equiv 1$ for all $t\geq 0$, we would recover the Glauber dynamics for the Gibbs measure given by the Hamiltonian $H=\sum_{xy\in E}\Phi_{xy}$. In this case we could without loss of generality simply consider only the first coordinate $(\sigma_t)_{t\geq 0}$ which would indeed be a Markov chain. Throughout the paper, we will refer to this process as the \textit{$\Phi$-Glauber dynamics} on $G$. 

A classical example of $\Phi$ is the one corresponding to the Ising model, where $S=\{-1,+1\}$ and for $xy\in E$ and $\sigma\in S^V$ we have
$$\Phi_{xy}(\sigma(x),\sigma(y)) ~=\ -\beta\sigma(x)\sigma(y),$$
for some fixed $\beta\in(0,\infty)$. In this case, if the clock corresponding to site $x$ rings at time $t$, we resample the value of $\sigma_t(x)$ to $\pm 1$ with probability proportional to
$$\exp\!\Big(\!\pm\beta\sum_{y:xy\in E}\eta_t(xy)\sigma_t(y)\Big).$$

We are particularly interested in the process $(\sigma_t,\eta_t)_{t\geq 0}$ defined on a sequence of graphs $(G_N)_{N\geq 1}$, where $G_N$ is the $d$-dimensional torus of size $N$ (for some fixed $d\geq 1$), i.e., the vertex set is given by $\T_N^d:=(\Z/N\Z)^d$ and the edge set $E(\T_N^d)$ is taken to be the set of nearest-neighbour edges. Note that in this case, a single collection of potentials $\Phi$ will be replaced by a sequence $(\Phi^{(N)})_{N\geq 1}$, where $\Phi^{(N)}=\{\Phi_{e}^{(N)}:e\in E(\T_N^d)\}$; we will assume no level of consistency between them \textit{a priori}. Moreover, we will assume that $p<p_c:=p_c(\Z^d)$, where $p_c$ is the critical probability for bond percolation in $\Z^d$. Thus, the environment typically consists of many small clusters. We will also consider the case when  $\lambda$ is small, so the environment evolves slowly. 

\paragraph{Mixing time and main result.}
Our main result concerns the mixing time of $(\sigma_t,\eta_t)_{t\geq 0}$ under the above assumptions on $p$ and $\lambda$, as well as some basic assumptions about the potentials. Given an irreducible aperiodic Markov chain $(X_t)_{t\geq 0}$ with stationary distribution $\pi$, we define the mixing time $t_\mathrm{mix}(\kappa)$ corresponding to $\kappa\in(0,1/2)$ as
$$t_\mathrm{mix}(\kappa) ~:=~ \inf\big\{t\geq 0: \max_x\|\Prm_x(X_t\in\cdot)-\pi\|_\TV\leq\kappa\big\},$$
where $\|\!\cdot\!\|_\TV$ denotes the total variation norm on measures, and $\Prm_x$ the law of $(X_t)_{t\geq 0}$ given $X_0=x$. For the case of classical Ising Glauber dynamics on a torus $\T_N^d$, $d\geq 2$, there exists $\beta_c=\beta_c(d)$, so that for $N$ large,
\begin{equation}\label{eq:ising_phase_transition}
t_\mix(\kappa) ~=~ \begin{cases}
\Theta(\log N), ~&\forall \beta<\beta_c,\\
\exp(\Theta(N^{d-1})), ~&\forall\beta>\beta_c;
\end{cases}
\end{equation}
see for example Chapter 15 in \cite{Levin_Peres_2017}. It is natural to ask whether the same phase transition would occur if we consider the Ising model on subcritical dynamical percolation as we defined above. 

Another important phenomenon related to mixing of Markov chains is that of cutoff. We say that a sequence of Markov chains with mixing times $t_\mix^{(1)},t_\mix^{(2)},\ldots$ exhibits a \textit{cutoff} if for all $\kappa\in(0,1)$, 
$$\lim_{n\to\infty}\frac{t_\mix^{(n)}(\kappa)}{t_\mix^{(n)}(1-\kappa)} ~=~ 1.$$
Most commonly, the sequence is taken to be a system (for example Ising Glauber dynamics) on an increasing sequence of graphs $G_1,G_2,\ldots$, which could for example be tori of increasing size. Moreover, we say that the cutoff has a window of size $O(w_n)$, if $w_n=o(t_\mix^{(n)})$ and 
\begin{align*}
    \lim_{\gamma\to\infty}\liminf_{n\to\infty}\max_x\|\Prm_x(X_{t_\mix^{(n)}-\gamma w_n}\in\cdot)-\pi\|_\TV ~&=~ 1, \\
    \lim_{\gamma\to\infty}\limsup_{n\to\infty}\max_x\|\Prm_x(X_{t_\mix^{(n)}+\gamma w_n}\in\cdot)-\pi\|_\TV ~&=~ 0.
\end{align*}
Despite the cutoff phenomenon being conjectured for a large class of models, it remains unproved for many classical cases. A major breakthrough in the topic was a series of papers by Lubetzky and Sly, who established in \cite{Lubetzky_Sly_2012} the cutoff for Ising Glauber dynamics for any $\beta<\beta_c(d)$, where $\beta_c(d)$ is the critical value of Ising model on $\Z^d$, at a window $O(\log\log n)$, which was later improved in \cite{Lubetzky_Sly_2015} to $O(1)$. 

Before stating our main result, we introduce some further notation. We will write $t_\mix^{(N)}(\kappa)$ for the $\kappa$-mixing time of the $\Phi^{(N)}$-spin system on dynamical percolation on $\T_N^d$, for some fixed $d\geq 1$. Moreover, given a connected subgraph $F$ of $\T_N^d$, we write $\tilde{t}_\mix^{(N)}(\kappa,F)$ for the $\kappa$-mixing time corresponding to $\Phi^{(N)}$-Glauber dynamics on $F$. Now we define, for $a>0$, 
$$h_\kappa^{(N)}(a) ~:=~ \sup_{F:|V(F)|\leq a} \tilde t_\mix^{(N)}(\kappa,F),$$
where the supremum is over connected subgraphs of $\T_N^d$; the supremum above is indeed finite as long as the chain is ergodic. Moreover, we write
\begin{equation}\label{eq:def_h}
h_\kappa(a) ~:=~ \sup_{N\geq 1} h_\kappa^{(N)}(a).
\end{equation}

\begin{theorem}\label{thm}
 Let $(\Phi^{(N)})_{N\geq 1}$ be such that $\Phi^{(N)}$ is a collection of potentials on $\T_N^d$ as above, and assume that for each $\kappa\in(0,1/2)$ and $a>0$,  we have $h_\kappa(a)<\infty$. Then, for any $d\geq 1$, $\kappa\in(0,1/2)$ and $p<p_c(d)$, there exist $C_\star>0$, $\lambda_0>0$ and $N_0\in\N$, so that for each $\lambda<\lambda_0$ and $N>N_0$,
    \begin{equation}\label{eq:theorem}
    \frac{1}{C_\star\lambda}\log N ~\leq~ t_\mix^{(N)}(\kappa) ~\leq~ \frac{C_\star}{\lambda}\log N.
    \end{equation}
    Moreover, if instead of $\lambda$ fixed we consider a sequence $\lambda_N\to 0$, the sequence $t_\mix^{(N)}(\kappa)$ exhibits a cutoff with a window $o(\log N)$.
\end{theorem}

\begin{remark}
    The reader should note that, in the theorem above, the threshold $\lambda_0$ does indeed depend on $\Phi$. For example, in the case of Ising model, it may depend on the inverse temperature $\beta$; in this particular case, the theorem above tells us that for any $\beta\geq 0$, we can find $\lambda_0$ sufficiently small so that the mixing time has order $\frac{\log N}{\lambda}$ whenever $\lambda<\lambda_0$.
\end{remark}

It is also natural to consider the version of the process started from stationary environment, that is, sampling $\eta_0$ according to $\P_p$. Writing $t_\mix^{\sta,(N)}(\kappa)$ for the mixing time of the corresponding Markov chain, we obtain the following result.

\begin{theorem}\label{thm_stationary}
    Let $(\Phi^{(N)})_{N\geq1}$ be as in Theorem \ref{thm}. Then, for any $d\geq 1$, $\kappa\in(0,1/2)$, $p<p_c(d)$ and $C^\star>0$, there exist $\lambda_0>0$, $N_0\in\N$, so that for each $\lambda<\lambda_0$ and $N>N_0$,
    $$t_\mix^{\sta,(N)}(\kappa)~\leq~\frac{C^\star}{\lambda}\log N.$$
\end{theorem}

It may be tempting to think that the above result implies the upper bound in Theorem~\ref{thm} by the following coupling argument. Start one configuration from the stationary distribution and the other from an arbitrary initial configuration. Noting that the dynamics on the environment is simply dynamical percolation, which is a biased random walk on the hypercube, the environments can be coupled in a time of order $\frac{1}{\lambda}\log N$; in fact, the dynamics of the environment alone has a cutoff. So one can perform a two-stage coupling, where in the first stage we couple the environments of the two processes during a time of order $\frac{1}{\lambda}\log N$, and then we perform a second coupling that couples the spins of the two processes. Since one configuration starts from stationarity, one may be inclined to believe that the second stage coupling can be obtained by the optimal coupling from the total variation mixing time starting from a stationary environment, which by Theorem~\ref{thm_stationary} can be carried out in time of order $\frac{1}{\lambda}\log N$. However, there are dependences between the spin configuration and the environment obtained at the end of the first stage. We will deal with those via a different notion of mixing time from stationarity. This will be explained in Section~\ref{sec:prelim}, with the corresponding version of  Theorem~\ref{thm_stationary} for this different notion of mixing time from stationarity being stated in Proposition~\ref{lemma:main_lemma_section_3}.

\begin{remark}
    Under some additional assumptions on $(\Phi^{(N)})_{N\geq 1}$, we also obtain a lower bound on $t_\mix^{\sta,(N)}$, for large enough $N$, which is of order $\log N$ as well. The result is stated precisely in Theorem \ref{thm:lower_bound}; the statement is deferred to Section \ref{sec:lower_bound}, in order to avoid introducing further technical notions in this section. The significance of this result is the following. Theorem~\ref{thm} establishes cutoff when $\lambda=\lambda_N\to0$ as $N\to\infty$. The reason for the cutoff in this case is that the environment dynamics has a cutoff and the mixing time of the spin component starting from a stationary environment is of order smaller than the mixing of the environment. However, we are left with the question of weather there is cutoff in the case of a fixed $\lambda$ that does not depend on $n$. Cutoff in this case would follow if the mixing time from a stationary environment was of order $o(\log N)$. Theorem \ref{thm:lower_bound} tells us that this is not the case for a large class of potentials. This does not rule out the possibility of cutoff, but it suggests the need of a new idea to establish cutoff in this regime of $\lambda$.
\end{remark}

\paragraph{Discussion of the proof.} The proof of the lower bound simply relies on the fact that  the mixing time of $(\sigma_t,\eta_t)_{t\geq 0}$ is bounded from below by the mixing time of $(\eta_t)_{t\geq 0}$, which is known to be of order $\lambda^{-1}\log N$, due to i.i.d.~nature of the updates. 

The difficult part is the upper bound, especially since many of the classic techniques (such as the spectral gap) are not readily available, due to the lack of reversibility and limited knowledge about the invariant measure $\pi$. Instead, we use the coupling technique, which is based on constructing a coupling of two copies $(\sigma_t,\eta_t)_{t\geq 0},(\sigma_t',\eta_t')_{t\geq 0}$ of the process; the upper bound on the mixing time can be obtained by studying the probability (under this coupling) that $\sigma_t\neq\sigma_t'$, $\eta_t\neq\eta_t'$ for $t$ large. 

We design a coupling to exploit that the environment typically evolves more slowly than the spin component. This is achieved by first dividing the time $[0,\infty)$ into disjoint intervals of length $\varepsilon/\lambda$ (for some small $\varepsilon$), which are short from the perspective of the environment but long from the perspective of the spin component, which is achieved by taking $\lambda\ll\varepsilon$. In particular, given one such interval $[a,b)$ of length $\varepsilon/\lambda$, the value of $\varepsilon$ is set in a way to obtain the following:
\begin{itemize}
    \item[(i)] for any typical cluster $\C$ of $\eta_a$, which has size $O(1)$, it is unlikely that any edge adjacent to $\C$ changes its value during $[a,b)$, and in particular, the cluster remains unchanged throughout the time interval with high probability;
    \item[(ii)] the spin value at any particular site is likely to be ``updated'' many times during $[a,b)$.
\end{itemize}

The analysis is simplified by the fact that we can obtain a bound on mixing (see Section \ref{sec:prelim}) by studying a version of the process, where we take $\eta_t=\eta_t'$ for all $t\geq 0$. This allows us to inspect, at the start of each time interval $[a,b)$, the disagreement between $\sigma_a$ and $\sigma_a'$ on sites associated to each cluster of $\eta_a$. Due to (i), the environment in most of those clusters will remain the same during the time interval. Our coupling is constructed so that in clusters for which this is true, we obtain the following:
\begin{itemize}
    \item if there was agreement on the sites of the cluster at time $a$ between $\sigma_a$ and $\sigma_a'$, this agreement is preserved until time $b$, and
    \item even if there was disagreement on at least one site of the cluster at time $a$, there will be agreement at time $b$ with high probability.
\end{itemize}
The main difficulty comes from clusters on which there is disagreement at time $a$ and whose environment changes at the boundary during the time interval, allowing disagreement to propagate to other sites. We need to control how much disagreement can propagate when such bad events happen. 

We solve this issue by dividing the space-time slab $\T_N^d\times[0,\infty)$ into (partially overlapping) space-time boxes, of side-length $L$ in the spatial dimension and $C/\lambda$ in the temporal dimension, where $L$ is a large enough value and $C$ is large enough for the slab to contain many $\varepsilon/\lambda$-intervals. Furthermore, we define a notion of \textit{ajar} clusters, which tell us how far a disagreement can spread from a given site within a $\varepsilon/\lambda$-interval in the worst-case scenario; showing that those are typically small enough is of major importance in overcoming the aforementioned issue. We then define a notion of a \textit{good} (resp.~\textit{bad}) box, so that 
\begin{itemize}
    \item[(a)] the event that a given box is good depends only on the local updates (i.e., on updates to sites and edges in the box, during the duration of the box), and
    \item[(b)] for large enough $t$, we are able to give a sufficient condition for agreement between $\sigma_t$ and $\sigma_t'$ on the entirety of the torus purely in terms of the existence of a time-oriented path of good boxes from time $0$ to time $t$. 
\end{itemize}
The major steps are proving that a given box is good with high probability and defining good boxes carefully enough so that the above properties hold.

\paragraph{Related works.}

Dynamical percolation was first defined in the late 90s by H\"aggstr\"om, Peres and Steif \cite{Haggstrom_Peres_Steif_1997} and independently by Itai Benjamini; for a nice exposition see \cite{Steif2009}. A model related to ours is that of a random walk on dynamical percolation, which was introduced by Peres, Stauffer and Steif \cite{Peres_Stauffer_Steif_2014}. Variety of results followed, for example \cite{Peres_Sousi_Steif_2018, Peres_Sousi_Steif_2019, Hermon_Sousi_2020, Sousi_Thomas_2020, Markering_2024, Gu_Jiang_Peres_Shi_Hao_Yang_2024a, Gu_Jiang_Peres_Shi_Hao_Yang_2024b}. More recently, some modifications of the model have also been explored, namely the biased random walk on dynamical percolation in \cite{Andres_Gantert_Schmid_Sousi_2024, Olzhabayev_Schmid_2025}, and simple random walk on dynamical random cluster in \cite{Lelli_Stauffer_2024, Galanis_Goldberg_Mifsud_2026}. Interacting particle systems on dynamical percolation have only been considered in the past few years. The contact process on dynamical percolation was first studied by Linker and Remenik \cite{Linker_Remenik_2020} and studied further in \cite{Hilario_Ungaretti_Valesin_Vares_2022, Seiler_Sturm_2023, Deshayes_Marchand_2026}. The voter model on dynamical percolation was introduced by Astoquillca \cite{Astoquillca_2026}.

There is also an existing literature studying dynamics of spin system on fixed random environments, though less related to this paper, so we will not attempt to do its span justice. As a short but diverse list of examples, one could see \cite{Cesi_Maes_Martinelli_1997, Dommers_denHollander_Jovanovski_Nardi_2017, Blanca_Gheissari_2023, Bovier_denHollander_Marello_Pulvirenti_Slowik_2024}.


\section{Preliminaries}\label{sec:prelim}

In this section we give some preliminaries on the mixing times of Glauber dynamics of spin systems. We also give a simple proof of the lower bound in Theorem \ref{thm} and give a preliminary upper bound on $t_\mix^{(N)}(\kappa)$, which we use in Section \ref{sec:proof} to prove the upper bound in Theorem \ref{thm}.

\paragraph{Total variation norm and coupling of random variables.}
Let $\mu,\nu$ be two probability measures on a measurable space $\Omega$. The \textit{total variation distance} between $\mu$ and $\nu$ is defined by
$$\|\mu-\nu\|_\TV ~:=~ \sup_{A\subseteq\Omega}|\mu(A)-\nu(A)|,$$
where the supremum is over measurable subsets. When $\Omega$ is finite, one more commonly employs the equivalent characterization
$$\|\mu-\nu\|_\TV ~=~ \frac{1}{2}\sum_{\omega\in\Omega}|\mu(\omega)-\nu(\omega)|.$$
Recall that given two random variables $X$ and $Y$ defined on $\Omega$ with respective laws $\mu$ and $\nu$, a \textit{coupling} between $X$ and $Y$ is a probability measure $\PP$ on $\Omega\times\Omega$, whose marginal distributions are precisely $\mu$ and $\nu$. It is straightforward to verify (see for example the proof of Proposition 4.2 in \cite{Georgii_Haggstrom_Maes_1999}) that given $X,Y$ as above and any coupling $\PP$ between them,
\begin{equation}\label{eq:coupling_inequality}
\|\mu-\nu\|_\TV ~\leq~ \PP(X\neq Y).
\end{equation}
A coupling that attains the equality, i.e., $\|\mu-\nu\|_\TV=\PP(X\neq Y)$, is called the \textit{optimal coupling}. It is a classic result that such a coupling always exists on nice enough spaces -- for example, it suffices to assume that $\Omega$ is a Polish space.

\paragraph{Mixing times of Markov chains.} Let $(X_t)_{t\geq 0}$ be a continuous time Markov chain on a finite state space $\Omega$ and write $\Prm_x^{t}$ for the law of $X_t$ conditional on $X_0=x$. We also write $\Prm_x$ for the law of the entire chain, conditional on $X_0=x$, so that $\Prm_x^t(\cdot)=\Prm_x(X_t\in\cdot)$. If $(X_t)_{t\geq 0}$ is ergodic (that is, recurrent and irreducible), then there exists a unique measure $\pi$ that is invariant with respect to the dynamics of $(X_t)_{t\geq 0}$; moreover, the process converges weakly to $\pi$ as $t\to\infty$, i.e.,
$$X_t ~\xrightarrow{(d)}~ \pi.$$
To quantify the speed of this convergence we define the notion of the \textit{mixing time}:
$$t_\mix(\kappa) ~:=~  \inf\!\Big\{t\geq0:\max_{x\in\Omega}\|\Prm_x^t-\pi\|_\TV\leq\kappa\Big\}, \quad\kappa\in(0,1/2).$$
Classically, one writes $t_\mix:=t_\mix(1/4)$, however this particular choice of $\kappa$ is quite arbitrary as it turns out that 
$$t_\mix(\kappa) ~\leq~ \lceil\log_2(\kappa^{-1})\rceil t_\mix.$$

A classical way to approach obtaining an upper bound on $t_\mix(\kappa)$ utilizes precisely the notion of coupling. Given a Markov chain generator $\mathcal{L}$, a coupling of Markov chains with generator $\mathcal{L}$ is a law $\PP$ on processes $(X_t,Y_t)_{t\geq 0}$ so that the marginal law of both $(X_t)_{t\geq 0}$ and $(Y_t)_{t\geq 0}$ is one of a Markov chain with generator $\mathcal{L}$. In particular, letting $\PP_{\!x,y}$ be a coupling of Markov chains starting from $x$ and $y$ respectively, then $\PP_{\!x,y}((X_t,Y_t)\in\cdot)$ is a coupling of $\Prm_x^t$ and $\Prm_{y}^t$, which implies that
$$\|\Prm_x^t-\Prm_y^t\|_\TV ~\leq~ \PP_{\!x,y}(X_t\neq Y_t).$$
A simple application of the inequality 
$$\|\Prm_x^t-\pi\|_\TV ~\leq~ \max_{y\in\Omega}\|\Prm_x^t-\Prm_y^t\|_\TV$$
(see for example Lemma 4.10 in \cite{Levin_Peres_2017}) then yields that
$$t_\mix(\kappa) ~\leq~ \inf\Big\{t\geq 0: \max_{x,y\in\Omega}\PP_{\!x,y}(X_t\neq Y_t)\leq\kappa\Big\}.$$
Note that given $t\geq 0$ fixed, one can construct a coupling $\PP_{\!x,y}^t$ of $(X_t)_{t\geq 0}$ and $(Y_t)_{t\geq 0}$ started from $x$ and $y$, such that
$$\PP_{\!x,y}^t(X_t\neq Y_t) ~=~ \|\Prm_x^t-\Prm_y^t\|_\TV.$$
We will call such coupling the \textit{$t$-optimal coupling}.

\paragraph{Preliminary bounds on $t_\mix$ and stationary environment.}

We conclude this section by providing a proof of the lower bound in Theorem \ref{thm} and as well as a preliminary upper bound on $t_\mix(\kappa)$ which simplifies the proof of the upper bound in Theorem \ref{thm}. 

Let $\Psf_{\eta}^t$ denote the law of $\eta_t$ conditional on $\eta_0=\eta$. It is trivial to check (note that dynamical percolation can be understood as a random walk on a hypercube) that the mixing time of the environment $(\eta_t)_{t\geq 0}$, defined by
$$t_\mix^\env(\kappa) ~:=~ \inf\Big\{t\geq 0: \max_{\eta}\|\Psf_\eta^t-\P_p\|_\TV\leq\kappa\Big\},$$
where the maximum is over $\eta\in\{0,1\}^{E(\T_N^d)}$, is of order $\lambda^{-1}\log N$. Thus, the following Lemma immediately implies the lower bound in Theorem \ref{thm}.
\begin{lemma}\label{lemma:lower_bound}
    For any choice of $\kappa\in(0,1/2)$, letting $t_\mix(\kappa)$ denote the mixing time of $(\sigma_t,\eta_t)_{t\geq 0}$,
    $$t_\mix^\env(\kappa) ~\leq~ t_\mix(\kappa).$$
\end{lemma}

\begin{proof}
    It is sufficient to show that for any choice of $\sigma\in S^{\T_N^d}$ and $\eta\in\{0,1\}^{E(\T_N^d)}$,
    $$\|\Psf_\eta^t-\P_p\|_\TV ~\leq~ \|\Prm_{\sigma,\eta}^t-\pi\|_\TV.$$
    Writing $\Pi_2(\sigma,\eta)=\eta$, we have that
    $$\Psf_\eta^t ~=~ \Prm_{\sigma,\eta}^t\circ\Pi_2^{-1} \quad\text{and}\quad \P_p ~=~ \pi\circ\Pi_2^{-1},$$
    which follows from the fact that $(\eta_t)_{t\geq 0}$ does not depend on $(\sigma_t)_{t\geq 0}$. It thus follows that 
    \begin{align*}
        \|\Psf_\eta^t-\P_p\|_\TV ~&=~ \frac{1}{2}\sum_{\zeta}\big|\Prm_{\sigma,\eta}^t(\Pi_2^{-1}(\zeta))-\pi(\Pi_2^{-1}(\zeta))\big| \\
        &=~ \frac{1}{2}\sum_{\zeta}\Big|\sum_{\xi}\big(\Prm_{\sigma,\eta}^t(\xi,\zeta)-\pi(\xi,\zeta)\big)\Big| \\
        &\leq~ \frac{1}{2}\sum_{\xi,\zeta}\big|\Prm_{\sigma,\eta}^t(\xi,\zeta)-\pi(\xi,\zeta)\big| \\
        &=~ \|\Prm_{\sigma,\eta}^t-\pi\|_\TV.
    \end{align*}
\end{proof}

To obtain a preliminary upper bound on $t_\mix(\kappa)$ one can, exploit the fact that one can construct a coupling such that after $(\eta_t)_{t\geq 0}$ and $(\eta_t')_{t\geq 0}$ coalesce, they remain in agreement. This suggests that we can assume that the environment starts from stationarity. However, as we will see below, there are some difficulties with that approach which require an additional technical result. 

Given $\sigma\in S^{\T_N^d}$, we write $\Prm_\sigma^t(\cdot):=\sum_{\eta}\P_p(\eta)\Prm_{\sigma,\eta}^t(\cdot)$ for the law of $(\sigma_t,\eta_t)$ when started from spin configuration $\sigma$ and from stationary environment. Moreover, given a function $f:\{0,1\}^{E(\T_N^d)}\to S^{\T_N^d}$, write $\Prm_f^t(\cdot):=\sum_{\eta}\P_p(\eta)\Prm_{f(\eta),\eta}^t(\cdot)$ for the law of chain $(\sigma_t,\eta_t)$ started from stationary environment and a spin configuration that depends on the realization of the environment via $f$. The motivation for the latter is as follows. For fixed $t\geq 0$ and $f,g:\{0,1\}^{E(\T_N^d)}\to S^{\T_N^d}$, let $\PP_{\!f,g}^t$ be some coupling of $\Prm_f^t$ and $\Prm_g^t$ obtained by sampling $\eta\sim\P_p$ and then running some coupling of $\Prm_{f(\eta),\eta}^t$ and $\Prm_{g(\eta),\eta}^t$. Then,
$$\PP_{\!f,g}^t\big((\sigma_t,\eta_t)\neq(\sigma_t',\eta_t')\big) ~\geq~ \sum_{\eta}\P_p(\eta)\|\Prm_{f(\eta),\eta}^t-\Prm_{g(\eta),\eta}^t\|_\TV,$$
so in particular
$$\max_{f,g}\PP_{f,g}^t\big((\sigma_t,\eta_t)\neq(\sigma_t',\eta_t')\big) ~\geq~ \sum_{\eta}\P_p(\eta)\max_{\sigma,\sigma'}\|\Prm_{\sigma,\eta}^t-\Prm_{\sigma',\eta}^t\|_\TV,$$
using that maximizing over $f,g$ is equivalent to maximizing over $\sigma,\sigma'$ for each particular realization of $\eta$.
\begin{lemma}\label{lemma:upper_bound}
    Let $\PP_{\!f,g}$ be any coupling of the copies of the chain obtained by sampling $\eta\sim\P_p$ and given each realization running some coupling of $\Prm_{f(\eta),\eta}$ and $\Prm_{g(\eta),\eta}$. Then,
    $$t_\mix(\kappa) ~\leq~ t_\mix^\env(\kappa/2)+\inf\!\big\{t\geq 0:\max_{f,g}\PP_{\!f,g}\big((\sigma_t,\eta_t)\neq(\sigma_t',\eta_t')\big)\leq\kappa/2\big\},$$
    where the maximum is over $f,g:\{0,1\}^{E(\T_N^d)}\to S^{\T_N^d}$.
\end{lemma}

\begin{proof}
    Let $T_1=t_\mix^\env(\kappa/2)$ and 
    $$T_2~=~\inf\!\Big\{t\geq 0:\max_{f,g}\sum_{\eta}\P_p(\eta)\|\Prm_{f(\eta),\eta}^t-\Prm_{g(\eta),\eta}^t\|_\TV\leq\kappa/2\Big\}.$$
    It is sufficient to prove that for $T=T_1+T_2$, the following holds, for any choice of $\sigma,\eta$:
    \begin{equation}\label{eq:upper_bound_lemma}
    \|\Prm_{\sigma,\eta}^T-\pi\|_\TV ~\leq~ \|\Psf_\eta^{T_1}-\P_p\|_\TV + \max_{f,g}\sum_{\zeta}\P_p(\zeta)\|\Prm_{f(\zeta),\zeta}^{T_2}-\Prm_{g(\zeta),\zeta}^{T_2}\|_\TV.
    \end{equation}
    Now fix $t\geq 0$ and let $\PP_{\sigma,\eta,\pi}^{t}$ denote a coupling of $(\sigma_t,\eta_t)\sim\Prm_{\sigma,\eta}^t$ and $(\sigma_t',\eta_t')\sim\pi$, obtained by sampling $(\eta_t,\eta_t')$ according to the optimal coupling of $\Psf_\eta^t$ and $\P_p$ and then sampling $(\sigma_t,\sigma_t')$ according to the product measure $\Prm_{\sigma,\eta}^t(\cdot|\eta_t)\otimes\pi(\cdot|\eta_t')$. By definition of this coupling,
    $$\PP_{\sigma,\eta,\pi}^t(\eta_t\neq\eta_t') ~=~ \|\Psf_\eta^t-\P_p\|_\TV.$$
    Moreover, write $\PP_{\sigma,\sigma',\eta}^{t}$ for the coupling of $(\sigma_t,\eta_t)\sim\Prm_{\sigma,\eta}^t$ and $(\sigma_t',\eta_t')\sim\Prm_{\sigma',\eta}^t$ obtained by sampling $\eta_t\sim\Psf_{\eta}^t$, setting $\eta_t'=\eta_t$ and then sampling $(\sigma_t,\sigma_t')$ according to the optimal coupling of $\Prm_{\sigma,\eta}^t(\cdot|\eta_t)$ and $\Prm_{\sigma',\eta}^t(\cdot|\eta_t')$. One can see that under this coupling
    \begin{align*}
        \PP_{\sigma,\sigma',\eta}^t\big((\sigma_t,\eta_t)\neq(\sigma_t',\eta_t')\big) ~&=~ \PP_{\sigma,\sigma',\eta}^t(\sigma_t\neq\sigma_t') \\
        &=~ \sum_{\zeta}\Psf_\eta^t(\zeta)\|\Prm_{\sigma,\eta}^t(\cdot|\eta_t=\zeta)-\Prm_{\sigma',\eta}^t(\cdot|\eta_t=\zeta)\|_{\TV} \\
        &=~ \frac{1}{2}\sum_{\xi,\zeta}\big|\Prm_{\sigma,\eta}^t(\xi,\zeta)-\Prm_{\sigma',\eta}^t(\xi,\zeta)\big| \\
        &=~ \|\Prm_{\sigma,\eta}^t-\Prm_{\sigma',\eta}^t\|_\TV,
    \end{align*}
    i.e., $\PP_{\sigma,\sigma',\eta}^t$ is the optimal coupling of $\Prm_{\sigma,\eta}^t$ and $\Prm_{\sigma',\eta}^t$. In the third equality, we used $\Prm_{\sigma,\eta}^t(\zeta)=\Psf_{\eta}^t(\zeta)$, which holds because dynamical percolation is independent of the spin component. 

    We now use the above to give an upper bound on the left hand side in \eqref{eq:upper_bound_lemma}. Let $\widehat\PP^T$ be a coupling of $(\sigma_T,\eta_T)\sim\Prm_{\sigma,\eta}^T$ and $(\sigma_T',\eta_T')\sim\pi$ obtained by first sampling $((\sigma_{T_1},\eta_{T_1}),(\sigma_{T_1}',\eta_{T_1}'))$ according to coupling $\PP_{\sigma,\eta,\pi}^{T_1}$ and then 
    \begin{itemize}
        \item[(i)] if $\eta_{T_1}=\eta_{T_1}'$, sampling $((\sigma_{T},\eta_T),(\sigma_T',\eta_T'))$ according to $\PP_{\sigma_{T_1},\sigma_{T_1}',\eta_{T_1}}^{T_2}$, and
        \item[(ii)] otherwise, sample them with respect to some arbitrary coupling of $\Prm_{\sigma_{T_1},\eta_{T_1}}^{T_2}$ and $\Prm_{\sigma_{T_1}',\eta_{T_1}'}^{T_2}$. 
    \end{itemize}
    Then, $\|\Prm_{\sigma,\eta}^T-\pi\|_\TV$ is bounded from above by
    \begin{align*}
        &\widehat\PP^T\big((\sigma_{T},\eta_{T})\neq(\sigma_{T}',\eta_{T}')\big) \\
        &\leq~ \sum_{\xi,\xi',\zeta,\zeta'}\PP_{\sigma,\eta,\pi}^{T_1}\big((\xi,\zeta),(\xi',\zeta')\big)\big[\1_{\{\zeta\neq\zeta'\}}+\1_{\{\zeta=\zeta'\}}\PP_{\xi,\xi',\zeta}^{T_2}(\sigma_{T}\neq\sigma_T')\big] \\
        &=~ \PP_{\sigma,\eta,\pi}^{T_1}(\eta_{T_1}\neq\eta_{T_1}') + \sum_{\xi,\xi',\zeta}\PP_{\sigma,\eta,\pi}^{T_1}\big((\xi,\zeta),(\xi',\zeta)\big)\PP_{\xi,\xi',\zeta}^{T_2}(\sigma_{T}\neq\sigma_T') \\
        &=~ \|\Psf_\eta^{T_1}-\P_p\|_\TV + \sum_{\xi,\xi',\zeta}\PP_{\sigma,\eta,\pi}^{T_1}\big((\xi,\zeta),(\xi',\zeta)\big)\|\Prm_{\xi,\zeta}^{T_2}-\Prm_{\xi',\zeta}^{T_2}\|_\TV.
    \end{align*}
    Noting that under $\PP_{\sigma,\eta,\pi}^{T_1}$, (i) conditionally on $\eta_{T_1}'$, the variable $\eta_{T_1}$ does not depend on $\sigma_{T_1}'$, and (ii) conditionally on $\eta_{T_1}$, the variable $\sigma_{T_1}$ does not depend on $(\sigma_{T_1}',\eta_{T_1}')$, one can decompose the probability $\PP_{\sigma,\eta,\pi}^{T_1}((\xi,\zeta),(\xi',\zeta))$ as
    $$\P_p(\zeta)\pi(\xi'|\zeta)\PP_{\sigma,\eta,\pi}^{T_1}(\eta_{T_1}=\zeta|\eta_{T_1}'=\zeta)\PP_{\sigma,\eta,\pi}^{T_1}(\sigma_{T_1}=\xi|\eta_{T_1}=\zeta).$$
    It follows that
    \begin{align*}
        &\sum_{\xi,\xi',\zeta}\PP_{\sigma,\eta,\pi}^{T_1}\big((\xi,\zeta),(\xi',\zeta)\big)\|\Prm_{\xi,\zeta}^{T_2}-\Prm_{\xi',\zeta}^{T_2}\|_\TV \\
        &~\leq~ \sum_{\zeta}\P_p(\zeta)\max_{\tilde{\xi},\hat{\xi}}\|\Prm_{\tilde{\xi},\zeta}^{T_2}-\Prm_{\hat{\xi},\zeta}^{T_2}\|_\TV\sum_{\xi}\PP_{\sigma,\eta,\pi}^{T_1}(\sigma_{T_1}=\xi|\eta_{T_1}=\zeta)\sum_{\xi'}\pi(\xi'|\zeta) \\
        &~=~ \sum_{\zeta}\P_p(\zeta)\max_{\xi,\xi'}\|\Prm_{\xi,\zeta}^{T_2}-\Prm_{\xi',\zeta}^{T_2}\|_\TV \\
        &~=~\max_{f,g}\sum_{\zeta}\P_p(\zeta)\|\Prm_{f(\zeta),\zeta}^{T_2}-\Prm_{g(\zeta),\zeta}^{T_2}\|_\TV,
    \end{align*}
    establishing \eqref{eq:upper_bound_lemma} and hence concluding the proof.
\end{proof}

Having obtained this bound, the mixing result in Theorem \ref{thm} can be obtained constructing a suitable collection of couplings $\PP_{\!f,g}=\PP_{\!f,g}^{N,\lambda}$ (corresponding to the system on $\T_N^d$ at environment speed $\lambda$) as above and a sequence $b_N(\lambda)=O(\lambda^{-1}\log N)$ with
$$\inf\!\big\{t\geq 0:\max_{f,g}\PP_{\!f,g}^{N,\lambda}(\sigma_t\neq\sigma_t')\leq\kappa/2\big\} ~\leq~ b_N(\lambda).$$
Since dynamical percolation undergoes a cutoff with window $O(1)$, we obtain the cutoff result in Theorem \ref{thm} if $b_N(\cdot)$ is such that, if fixed $\lambda$ is replaced by $\lambda_N\to 0$, we also have that $b_N(\lambda_N)=o(\lambda_N^{-1}\log N)$. For the latter to hold, it is sufficient to have $b_N(\lambda_N)=o(\lambda_N^{-1}\log N/|\log(1-\rho_N)|)$ for some $\rho_N\to 1$.


\paragraph{Some miscellaneous notation}

Before moving on to proving Theorem \ref{thm}, we introduce and comment on some notation and terminology that we will use throughout Section \ref{sec:proof}. 

Given some percolation configuration, we shall write $\C(x)$ for the connected component containing site $x$. In particular, given dynamical percolation $(\eta_t)_{t\geq 0}$ and fixed time $t\geq 0$, we shall write $\C_t(x)$ for the connected component of $\eta_t$ containing site $x$. 

Given a configuration $\sigma\in S^{\T_N^d}$ and a subset $\Delta\subseteq\T_N^d$, we will write $\sigma(\Delta):=(\sigma(x):x\in \Delta)\in S^\Delta$ for the restriction of $\sigma$ to $\Delta$. Moreover, given some percolation cluster $\C=(V(\C),E(\C))$, we will abuse the notation and write $\sigma(\C)$ for the restriction of $\sigma$ to $V(\C)$. On a similar note, given $\sigma,\sigma'\in S^{\T_N^d}$, we will say that $\sigma$ and $\sigma'$ agree on $\C$ or write $\sigma(\C)=\sigma'(\C)$ if $\sigma(x)=\sigma'(x)$ for all $x\in V(\C)$. 

Throughout the following section, we deal with a variety of different graphs and distances on them. Recall from Section \ref{sec:intro} that we write $d_G(\cdot,\cdot)$ for a graph distance on graph $G$. The underlying graph for the spin system is the discrete torus $(\T_N^d,E(\T_N^d))$, whose graph distance we simply denote by $\dist(\cdot,\cdot)$; we also write
$$\dist(x,\Delta) ~:=~ \min\{\dist(x,y):y\in\Delta\}$$
to denote the distance between a site $x\in\T_N^d$ and some subset $\Delta\subseteq\T_N^d$. To prove Theorem \ref{thm}, we define a collection of overlapping space-time boxes, which will be indexed in a set $\T_M^d\times\N$, where $\T_M^d$ is some smaller torus indexing spatial boxes. First, by abuse of notation, we define the following the following two distances on $\T_M^d$: for $i=(i_1,\ldots,i_d)$ and $j=(j_1,\ldots,j_d)$ on $\T_M^d$, we will write
\begin{align*}
    \|i-j\|_1 ~&:=~ \sum_{k=1}^d d_{\T_M}(i_k,j_k), \\
    \|i-j\|_\infty ~&:=~ \max_{1\leq k\leq d}d_{\T_M}(i_k,j_k),
\end{align*}
where $\T_M=\Z/M\Z$ is a one-dimensional torus. The reader should note that neither of the distances defined above is actually a norm, since we are on a torus. Moreover, by further abuse of notation, we define an analogue of those on the set $\T_M^d\times\N$ via
\begin{align*}
    \|(i,n)-(j,m)\|_1 ~&:=~ \|i-j\|_1+|n-m|, \\
    \|(i,m)-(j,m)\|_\infty ~&:=~ \max\{\|i-j\|_\infty,|n-m|\},
\end{align*}
where $i,j\in\T_M^d$ and $n,m\in\N$. As seen above, it will be clear whether either of the distances defined  above is considered on $\T_M^d$ or $\T_M^d\times\N$ from the way we denote the entries. Towards the end of Section \ref{sec:proof} we will also use a notion of $\ell^p$-paths, $p\in\{1,\infty\}$: we say that a collection $i_1,\ldots,i_k\in\T_M^d$ is an $\ell^1$-path (resp.~$\ell^\infty$-path) if for any $j=1,\ldots,k-1$, we have $\|i_j-i_{j+1}\|_1=1$ (resp.~$\|i_j-i_{j+1}\|_\infty=1$). An analogue definition extends to a collection $(i_1,n_1),\ldots,(i_k,n_k)$ on $\T_M^d\times\N$. 

Throughout the paper, we write $\log(\cdot)$ to denote the natural logarithm.


\section{Proof of Theorems \ref{thm} and \ref{thm_stationary}}\label{sec:proof}

In this section, we complete the proof of Theorem \ref{thm} and Theorem \ref{thm_stationary}. In particular, we provide the proof of the upper bound in \eqref{eq:theorem}, having already obtained the lower bound via Lemma \ref{lemma:lower_bound}. In light of Lemma \ref{lemma:upper_bound}, it is sufficient to prove the following statement.

\begin{proposition}\label{lemma:main_lemma_section_3}
    Let $(\Phi^{(N)})_{N\geq 1}$ be as in Theorem \ref{thm} and let $d\geq1$, $\kappa\in(0,1/2)$ and $p<p_c(d)$ be fixed. For any choice of $f,g:\{0,1\}^{E(\T_N^d)}\to S^{\T_N^d}$, there exists a coupling $\Q_{f,g}$ of chains $(\sigma_t,\eta_t)_{t\geq 0}$ and $(\sigma_t',\eta_t')_{t\geq 0}$ started from $(\sigma_0,\eta_0)=(f(\eta),\eta)$ and $(\sigma_0',\eta_0')=(g(\eta),\eta)$ respectively, where $\eta\sim\P_p$, such that for each $C^\star>0$ there exist $\lambda_0>0$, $N_0\in\N$, so that for all $\lambda<\lambda_0$ and $N> N_0$,
    $$\inf\!\big\{t\geq 0:\max_{f,g}\Q_{f,g}\big((\sigma_t,\eta_t)\neq(\sigma_t',\eta_t')\big)\leq \kappa/2\big\}~\leq~ \frac{C^\star}{\lambda}\log N.$$
\end{proposition}

Equipped with Proposition~\ref{lemma:main_lemma_section_3}, the proof of the main results from Section \ref{sec:intro} become trivial.

\begin{proof}[Proof of Theorem \ref{thm}]
    Let first $\lambda>0$ be fixed. The lower bound in \eqref{eq:theorem} is an immediate corollary of Lemma \ref{lemma:lower_bound}. Combining Lemma \ref{lemma:upper_bound} and Proposition~\ref{lemma:main_lemma_section_3}, we obtain that for any $C^\star>0$,
    $$t_\mix(\kappa) ~\leq~ t_\mix^\env(\kappa/2)+\frac{C^\star}{\lambda}\log N$$
    for any $\lambda$ small enough and $N$ large enough. The upper bound in Theorem \eqref{eq:theorem} then follows from the fact that $t_\mix^\env(\kappa/2)$ is of order $\lambda^{-1}\log N$. When instead $\lambda_N\to 0$, the cutoff statement follows by invoking the same statements as above along with the fact that for each $N$ large enough, we can find $C_N^\star>0$, so that 
    $$\inf\big\{t\geq 0\}\max_{f,g}\Q\big((\sigma_t,\eta_t)\neq(\sigma_t',\eta_t')\big)\big\}~\leq~ \frac{C_N^\star}{\lambda_N}\log N,$$
    and $C_N^\star\to 0$ as $N\to\infty$.
\end{proof}

\begin{proof}[Proof of Theorem \ref{thm_stationary}]
    Given Proposition~\ref{lemma:main_lemma_section_3}, it is sufficient to argue that 
    $$\max_{\sigma}\|\Prm_\sigma^{t}-\pi\|_\TV~\leq~ \max_{f,g}\Q_{f,g}\big((\sigma_t,\eta_t)\neq(\sigma_t',\eta_t')\big), \quad \forall t\geq 0,$$
    where $\Prm_\sigma^t$ denotes the law of $(\sigma_t,\eta_t)$ when $\sigma_0=\sigma$ and $\eta_0\sim\P_p$. Indeed, it holds that
    $$\max_{\sigma}\|\Prm_\sigma^t-\pi\|_\TV ~\leq~ \max_{\sigma,\sigma'}\|\Prm_\sigma^t-\Prm_{\sigma'}^t\|_\TV ~\leq~ \max_{f,g}\|\Prm_{f}^t-\Prm_{g}^t\|_\TV,$$
    recalling that $\Prm_f^t(\cdot)=\sum_{\eta}\P_p(\eta)\Prm_{f(\eta),\eta}^t$. The result is completed by recalling that for any choice of $f,g$ and any coupling $\PP_{f,g}^t$ of $\Prm_f^t$ and $\Prm_g^t$, it holds that
    $$\|\Prm_f^t-\Prm_g^t\|_\TV ~\leq~ \PP_{f,g}^t\big((\sigma_t,\eta_t)\neq(\sigma_t',\eta_t')\big).$$
\end{proof}

The rest of the section is dedicated to proving Proposition~\ref{lemma:main_lemma_section_3} and proceeds as follows. First, in Section \ref{subsec:coupling} we construct the couplings for which the above inequality will be proven. Then, in Section \ref{subsec:boxes}, we divide the space-time slab into (overlapping) boxes and define an event of a certain box being considered \textit{good} (or \textit{bad}). Moreover, we show that under this coupling, the probability that any particular box is good is sufficiently large, uniformly in $f,g$. Lastly, in  Section \ref{subsec:sufficient_perc} we establish a sufficient condition for coalescence of the coupled copies of the process in terms of sufficient non-percolation of bad boxes.


\subsection{Construction of the coupling}\label{subsec:coupling}

We have four processes $\eta=(\eta_t)_{t\geq 0}$, $\eta'=(\eta_t')_{t\geq 0}$ $\sigma=(\sigma_t)_{t\geq 0}$ and $\sigma'=(\sigma_t')_{t\geq 0}$. While $\eta$ and $\eta'$ are each independent of everything else, $\sigma$ (resp.~$\sigma'$) does in fact depend on $\eta$ (resp.~$\eta'$). The goal of defining a coupling is creating suitable dependence between $(\sigma_t,\eta_t)_{t\geq 0}$ and $(\sigma_t',\eta_t')_{t\geq 0}$. To make this precise, we define the so-called \textit{graphical representation} of the process $(\sigma_t,\eta_t)_{t\geq 0}$. 

Firstly, the evolution of $\eta$ started from some starting configuration $\eta_0$ can be constructed by (i) sampling an independent Poisson processes on $[0,\infty)$ with intensity $\lambda$ for each edge, and (ii) decorating each mark with a value sampled from $\mathrm{Unif}[0,1]$ distribution. This yields a collection of random triplets $(E,T,U)\in E(\T_N^d)\times[0,\infty)\times[0,1]$, which we denote by $\upd(\eta)$. Enumerating it as $(E_i,T_i,U_i)_{i\geq 1}$ (so that $T_i\leq T_{i+1}$ for all $i\geq 1$), we construct the process by starting with $\eta_0$ and updating it at each $T_i$, where we set the value of $\eta_{T_i}(E_i)$ to $1$ if $U_i\leq p$ and to $0$ if $U_i>p$. 

Similarly, one can construct the evolution of $\sigma$ started from $\sigma_0$ by (i) sampling an independent Poisson process on $[0,\infty)$ with intensity $1$ for each site, and (ii) decorating each mark with a uniform $[0,1]$ value, which again yields a collection of triplets $(V,T,U)\in\T_N^d\times[0,\infty)\times[0,1]$ which we denote by $\upd(\sigma)$. Enumerating it as before, we start from $\sigma_0$ and update it at each $T_i$, where we set the value of $\sigma_{T_i}(V_i)$ to $s_j$, $j=1,\ldots,|S|$, if
$$\sum_{k=1}^{j-1}\mu_{\sigma_{T_i},\eta_{T_i}}^{V_i}(s_k) ~\leq~ U ~<~ \sum_{k=1}^j\mu_{\sigma_{T_i},\eta_{T_i}}^{V_i}(s_k),$$
where $\mu_{\sigma_{T_i},\eta_{T_i}}^{V_i}$ is defined in \eqref{eq:update_conditional_probability}.

Given $\sigma$, defined as above, one can construct a copy of the process $\sigma'$ which depends on $\sigma$, by simply making its corresponding update collection $\upd(\sigma')$ a function of $\upd(\sigma)$. The simplest case of that is the so-called \textit{identity coupling}, where one simply lets $\upd(\sigma')=\upd(\sigma)$. The same type of construction can be carried out for the process $\eta$. However, this coupling alone will not help us obtain $\sigma_t=\sigma_t'$ for some time $t$. To achieve this, we will employ a more involved coupling. 

To construct the coupling for a given, which we shall denote by $\Q_{f,g}$, we start by sampling $\eta_0\sim\P_p$ and $\upd(\eta)$ according to the graphical representation above, and setting $\eta_0'=\eta_0$ and $\upd(\eta')=\upd(\eta)$. This corresponds to applying the identity coupling to both the starting configuration and the evolution of the two copies of dynamical percolation, started from stationarity. Moreover, we set $\sigma_0=f(\eta_0)$ and $\sigma_0'=g(\eta_0')$. We then proceed by partitioning the time axis $[0,\infty)$ into disjoint intervals $\I_1,\I_2,\ldots$ of length $\varepsilon/\lambda$, for some appropriate $\varepsilon$ which is to be decided at a later point. To be precise, we define
$$\I_k ~:=~ [t_{k-1},t_k), \quad t_k~:=~\frac{k\varepsilon}{\lambda}.$$
Now we define the coupling. At each time stamp $t_k$, $k\geq 0$, we inspect each cluster $\C$ of configuration $\eta_{t_k}$ individually, and proceed as follows.
\begin{itemize}
    \item[(a)] If $\sigma_{t_k}(x)=\sigma_{t_k}'(x)$ for all $x\in V(\C)$, then we impose identity coupling on $\C$ until time $t_{k+1}$. That is, updates of $\sigma'$ within $\I_{k+1}=[t_k,t_{k+1})$ corresponding to sites in $\C$ will be given by the sample triplets $(V,T,U)\in\upd(\sigma)$ with $T\in\I_{k+1}$ and $V\in V(\C)$.
    \item[(b)] Otherwise, we consider two cases. We look at all edge updates $(E,T,U)\in\upd(\eta)$ with $T\in\I_{k+1}$ and $E\in E^+(\C)$:
    \begin{itemize}
        \item[(b1)]  If there are none (and hence the cluster $\C$ remains unchanged for entirety of $\I_{k+1}$), the process  
  $(\sigma_t(\C))_{t\in\I_{k+1}}$ is just Glauber dynamics on a graph $\C$, independent of the process outside $\C$. In this case, we sample $(\sigma_t(\C),\sigma_t'(\C))_{t\in\I_{k+1}}$ according to the $\varepsilon/\lambda$-optimal coupling for $\Phi^{(N)}$-Glauber dynamics on $\C$ given starting configurations $\sigma_{t_k}(\C)$ and $\sigma_{t_k}'(\C)$.
        \item[(b2)] If there are edge updates to $\C$ within $\I_{k+1}$, then we simply run identity coupling until $t_{k+1}$. 
    \end{itemize}
\end{itemize}

The verification that this indeed defines a coupling is deferred to Appendix \ref{app:coupling}, as it requires a notion defined later in this section. 

One should also note that under this coupling, if $\sigma_{t_k}(x)=\sigma_{t_k}'(x)$ for all $x\in\T_N^d$ for some $k\geq 0$, then identity coupling will be implemented on all clusters, which will result in having $\sigma_t=\sigma_t'$ for all $t\geq t_k$. 


\subsection{Renormalization: good and bad boxes}\label{subsec:boxes}

Having defined the coupling $\Q_{f,g}$ in the previous section, the goal now is to study the behavior of $\Q_{f,g}(\sigma_t\neq\sigma_t')$, for sufficiently large $t$. 

We renormalize the space-time $\T_N^d\times[0,\infty)$ into smaller space-time boxes, whose does not depend on $N$. 

We now define the space-times boxes. Let $L=L_\lambda$, such that
\begin{equation}\label{eq:L_assump}
L\to\infty\quad\text{and}\quad\frac{L}{\log(1/\lambda)}\to 0, \quad\text{as}~\lambda\downarrow 0,
\end{equation}
write $L':=8L/10$ and assume that $L'\in\N$  and that $L'\cdot M=N$ for some $M=M_\lambda\in\N$. We first partition $\T_N^d$ into $M^d$ disjoint (spatial) boxes of side length $L'$, denoted by $\Lambda_i^\core\subset\T_N^d$, $i\in\T_M^d$. More formally, given any $i=(i_1,\ldots,i_d)\in\T_M^d$, we define
$$\Lambda_i^\core ~:=~ \prod_{\ell=1}^d \{i_\ell L',\ldots,(i_\ell+1)L'-1\}.$$
Furthermore, define a (spatial) box $\Lambda_i\subset\T_N^d$ with side length $L$ as an $(L/10)$-thickening\footnote{The $r$-thickening of a set $\Lambda\subseteq\T_N^d$ is defined as the set $\{x:\dist_\infty(x,\Lambda)\leq r\}$. In particular, if $\Lambda=\{-k,\ldots,k\}^d$, then its $r$-thickening is $\{-(k+r),\ldots,k+r\}^d$.} of $\Lambda_i^\mathrm{core}$. We will refer to $\Lambda_i^\core$ as the \textit{core} of $\Lambda_i$.

Fixing a constant $C>0$ and writing
$$\II_n ~=~ \II_{n}^{\lambda,C} ~:=~ \left[\frac{(n-1)C}{\lambda},\frac{(n+1)C}{\lambda}\right), \quad n\in\N,$$
we define the space-time boxes as
$$\B_{i,n} ~:=~ \Lambda_i\times\II_n, \quad i\in\T_M^d,n\in\N.$$
Notice that $\B_{i,n}\cap\B_{i',n'}=\emptyset$ iff $\|i-i'\|_\infty> 1$ or $|n-n'|> 2$. Moreover, the second half in the temporal dimension of $\B_{i,n}$ is precisely the first half of $\B_{i,n+1}$.

To cope with the fact that given $x\in\T_N^d$, the cluster containing $x$ changes with time, we introduce a more robust notion of a cluster. Given $k\geq 1$, we say that an edge $e$ is \textit{ajar} during $\I_k$, if there exists $t\in\I_k$ such that $\eta_t(e)=1$. This induces a natural configuration $\zeta_k\in\{0,1\}^{E(\T_N^d)}$ given by
$$\zeta_k(e) ~:=~ \1\{e~\text{is ajar during}~\I_k\}, \quad e\in E(\T_N^d).$$
It is immediate that for each $t\in\I_k$, $\eta_t\leq\zeta_k$ in the pointwise sense, which in turn implies that for any choice of $x\in\T_N^d$,
$$V(\C_t(x)) ~\subseteq~ V(\C_k^\ajar(x)), \quad \forall t\in\I_k,$$
where $\C_k^\ajar$ refers to $\zeta_k$-clusters. Note that for any $y\notin V(\C_k^\ajar(x))$, there is no $t\in\I_k$ and $x'\in V(\C_k^\ajar(x))$ for which $\eta_t(x',y)=1$, so we find that the restrictions of $(\sigma_t)_{t\in[a,b]}$ to $V(\C_k^\ajar(x))$ and $\T_N^d\setminus V(\C_k^\ajar(x))$ are independent. It is not difficult to see (as demonstrated in the proof of Lemma \ref{lemma:P(box_good)} below) that by taking $\varepsilon$ sufficiently small, the law of $\zeta_k$ can be stochastically dominated by subcritical Bernoulli percolation, which again tells us that typically, during $\I_k$, any two sites that are sufficiently distant from each other evolve independently. We use this fact to limit how far any disagreement that exists at time $t_{k-1}$ can spread during $\I_k$: note that indeed, if $x\neq y$ are such that
\begin{itemize}
    \item[(i)] $y\notin V(\C_k^\ajar(x))$,
    \item[(ii)] $\sigma_{t_{k-1}}$ and $\sigma_{t_{k-1}}'$ do not agree on $\C_k^\ajar(x)$, and
    \item[(iii)] $\sigma_{t_{k-1}}$ and $\sigma_{t_{k-1}}'$ agree on $\C_k^\ajar(y)$, 
\end{itemize}
then $\sigma_t$ and $\sigma_t'$ agree on $\C_k^\ajar(y)$ for all $t\in\I_{k}$; in particular, their disagreement on $\C_k^\ajar(x)$ does not affect the agreement on $\C_k^\ajar(y)$ during $\I_{k}$. 

Another consequence of $\eta_t$-clusters typically being small is, that by selecting $\varepsilon$ small, it is likely that no edge adjacent to some typical cluster receives an update during $\I_k$. This not only yields that this cluster coincides with the associated ajar cluster, but also that the restriction of $(\sigma_t)_{t\in\I_{k}}$ to this cluster (which remains static during $\I_k$) is just classic Glauber dynamics on a fixed graph. By then taking $\lambda$ sufficiently small (in order to make $\varepsilon/\lambda$ large), we can make the probability that $\sigma$ and $\sigma'$ reach agreement on this cluster (under the optimal coupling) arbitrarily close to $1$. This is central to our proof later, once we control the sizes of clusters.

We introduce some notation for subsets of $\upd(\sigma)$ and $\upd(\eta)$, contained in a specific space-time box. First we define the first and second half of the interval $\II_n$ via $\II_n^{(1)}:=[(n-1)C/\lambda,nC/\lambda)$  and $\II_n^{(2)}:=[nC/\lambda,(n+1)C/\lambda)$. Considering a box $\B_{i,n}$, $i\in\T_M^d,n\in\N$, we write for $k=1,2$,
\begin{align*}
    \upd_{i,n}^{(k)}(\sigma) ~&:=~ \{(V,T,U)\in\upd(\sigma):(V,T)\in\Lambda_i\times\II_n^{(k)}\}, \\
    \upd_{i,n}^{(k)}(\eta) ~&:=~ \{(E,T,U)\in\upd(\eta):(E,T)\in E(\Lambda_i)\times\II_n^{(k)}\},
\end{align*}
as well as 
\begin{align*}
    \upd_{i,n}(\sigma) ~&:=~ \upd_{i,n}^{(1)}(\sigma)\cup\upd_{i,n}^{(2)}(\sigma), \\
    \upd_{i,n}(\eta) ~&:=~ \upd_{i,n}^{(1)}(\eta)\cup\upd_{i,n}^{(2)}(\eta).
\end{align*}
We wish to define a notion of a box $\B_{i,n}$ being ``good'', which is measurable with respect to the $\sigma$-algebra generated by $\upd_{i,n}(\sigma)$, $\upd_{i,n}(\sigma')$ and $\upd_{i,n}(\eta)$. Before doing so, we introduce some further notation. We write $\overline{\Lambda_i^\core}$ for the $(L/100)$-thickening of $\Lambda_i^\core$. Note that following this definition, $\Lambda_i^\core\subset\overline{\Lambda_i^\core}\subset\Lambda_i$. Moreover, we write, for $k=1,2$,
\begin{itemize}
    \item $\mathcal{S}_{i,n}^{(k)}$ for the collection of locally finite subsets of $(\T_N^d\setminus\Lambda_i)\times\II_n^{(k)}\times[0,1]$, and
    \item $\mathcal{E}_{i,n}^{(k)}$ for the collection of locally finite subsets of $(E(\T_N^d)\setminus E(\Lambda_i))\times\II_n^{(k)}\times[0,1]$,
\end{itemize}
as well as 
\begin{align*}
    \mathcal{S}_{i,n} ~&:=~ \{\Xi_1\cup\Xi_2:\Xi_k\in\mathcal{S}_{i,n}^{(k)}\}, \\
    \mathcal{E}_{i,n} ~&:=~ \{\Upsilon_1\cup\Upsilon_2:\Upsilon_k\in\mathcal{E}_{i,n}^{(k)}\}.
\end{align*}

\begin{definition}[Good boxes]\label{def:good_boxes}
Let $i\in\T_M^d$ and $n\in\N$. We say that a box $\B_{i,n}=\Lambda_i\times\II_n$ is \textit{good} if the following conditions hold.
\begin{enumerate}[label=(A\arabic*), ref=A\arabic*]
    \item\label{A1} For any $\tilde\eta\in\{0,1\}^{E(\T_N^d)}$ and $\Upsilon\in\mathcal{E}_{i,n}$, the process $(\tilde\eta_t)_{t\in\II_n}$ started from $\tilde\eta$ and evolved using $\upd_{i,n}(\eta)\cup\Upsilon$ is such that for all $k$ with $\I_k\subset\II_n^{(2)}$ and all $x\in\Lambda_i$,
    $$|\tilde\C_k^\ajar(x)| ~\leq~ \log^2 L, $$
    where $\tilde\C_k^\ajar(x)$ refers to the $\I_k$-ajar cluster of $(\tilde\eta_t)_{t\geq 0}$ containing $x$.
    \item\label{A2} For any $\tilde\sigma,\hat\sigma\in S^{\T_N^d}$, $\tilde\eta\in\{0,1\}^{E(\T_N^d)}$, $\Xi,\Xi'\in\mathcal{S}_{i,n}^{(2)}$ and $\Upsilon\in\mathcal{E}_{i,n}^{(2)}$, the processes $(\tilde\sigma_t,\tilde\eta_t)_{t\in\II_n^{(2)}}$ and $(\hat\sigma_t,\tilde\eta_t)_{t\in\II_n^{(2)}}$, started from $(\tilde\sigma,\tilde\eta)$ and $(\hat\sigma,\tilde\eta)$, respectively, and evolved using $\upd_{i,n}^{(2)}(\sigma)\cup\Xi$ and $\upd_{i,n}^{(2)}(\eta)\cup\Upsilon$ for $(\tilde\sigma_t,\tilde\eta_t)_{t\in\II_n^{(2)}}$, and $\upd_{i,n}^{(2)}(\sigma')\cup\Xi'$ and $\upd_{i,n}^{(2)}(\eta)\cup\Upsilon$ for $(\hat\sigma_t,\tilde\eta_t)_{t\in\II_n^{(2)}}$, are such that 
    $$\tilde\sigma_\frac{(n+1)C}{\lambda}(x) ~=~ \hat\sigma_\frac{(n+1)C}{\lambda}(x), \quad \forall x\in\overline{\Lambda_i^\core}.$$
    \item\label{A3} For any $\tilde\sigma,\hat\sigma\in S^{\T_N^d}$ such that $\tilde\sigma(\overline{\Lambda_i^\core})=\hat\sigma(\overline{\Lambda_i^\core})$, $\tilde\eta\in\{0,1\}^{E(\T_N^d)}$, $\Xi,\Xi'\in\mathcal{S}_{i,n}^{(2)}$ and $\Upsilon\in\mathcal{E}_{i,n}^{(2)}$, the processes $(\tilde\sigma_t,\tilde\eta_t)_{t\in\II_n^{(2)}}$ and $(\hat\sigma_t,\tilde\eta_t)_{t\in\II_n^{(2)}}$ defined as in Condition (\ref{A2}) are such that for all $t\in\II_n^{(2)}$,
    $$\tilde\sigma_t(x) ~=~ \hat\sigma_t(x), \quad \forall x\in\Lambda_i^\core.$$
\end{enumerate}
If any of those conditions fail, we say that $\B_{i,n}$ is a \textit{bad} box.
\end{definition}

Note also, that by the nature of the events defined in Definition \ref{def:good_boxes}, their probabilities do not depend on the particular choice of $f,g$, and hence all the bounds will be uniform in $f,g$. Thus, for the rest of the section, we will suppress the subscript and simply write $\Q$ in place of $\Q_{f,g}$.

\begin{lemma}\label{lemma:P(box_good)}
    For an appropriate choice of $L= L_\lambda$ as in \eqref{eq:L_assump} and setting $\varepsilon=\log^{-5}L$, there exist $C_*,\lambda_0,c>0$ so that for any $(i,n)\in\T_M^d\times\N$, $C\geq C_\star$ and $\lambda\leq\lambda_0$,
    $$\Q\big(\B_{i,n}~\mathrm{is~good}\big) ~\geq~ 1-e^{-c\log^2 L}.$$
\end{lemma}

\begin{remark}
    We briefly explain the role of each condition in Definition \ref{def:good_boxes}. Condition (\ref{A1}) imposes a uniform bound on the range of interaction of any site during any time interval $\I_k$. The choice of $\log^2 L$ as the bound is fairly arbitrary -- what is important is that the bound is of order higher than $\log L$. Condition (\ref{A2}) ensures that at the end of the time of the box, $\sigma$ and $\sigma'$ agree on the entire $\Lambda_i^\core$ as well as all the clusters intersecting $\Lambda_i^\core$, assuming $L$ is sufficiently large. The role of Condition (\ref{A3}) is to ensure that if the box $\B_{i,n-1}$ was good, which implies that $\sigma$ and $\sigma'$ agree on $\overline{\Lambda_i^\core}$ at the starting time of $\II_n^{(2)}$, then $\sigma$ and $\sigma'$ agree on $\Lambda_i^\core$ throughout $\II_n^{(2)}$.
\end{remark}

\begin{proof}[Proof of Lemma \ref{lemma:P(box_good)}]
    Pointing out the obvious fact that 
    $$\text{(\ref{A1})}^c\cup\text{(\ref{A2})}^c\cup\text{(\ref{A3})}^c ~=~ \text{(\ref{A1})}^c\cup\big(\text{(\ref{A1})}\cap\text{(\ref{A2})}^c\big)\cup\big(\text{(\ref{A1})}\cap\text{(\ref{A3})}^c\big),$$
    and reminding the reader that $L\to \infty$ as $\lambda$ approaches $0$, we can divide the proof into showing that
    \begin{itemize}
        \item[(I)] there exist $C_\star,c>0$ so that for any $C\geq C_\star$ and $L$ large, 
        \begin{equation}\label{eq:bound_A1}
        \Q\big((\mathrm{\ref{A1}})^c\big) ~\leq~ \frac{1}{3}e^{-c\log^2 L};
        \end{equation}
        \item[(II)] letting $C_\star,c>0$ be as in (I), for $C\geq C_\star$ and $L$ large,
        \begin{equation}\label{eq:bound_A2}
        \Q\big((\text{\ref{A1}})\cap(\text{\ref{A2}})^c\big) ~\leq~ \frac{1}{3}e^{-c\log^2 L};
        \end{equation}
        \item[(III)] letting $C_\star,c>0$ be as in (I), for $C\geq C_\star$ and $L$ large,
        \begin{equation}\label{eq:bound_A3}
        \Q\big((\text{\ref{A1}})\cap(\text{\ref{A3}})^c\big) ~\leq~ \frac{1}{3}e^{-c\log^2 L}.
        \end{equation}
    \end{itemize}
    Without loss of generality, we assume that $C/\varepsilon\in\N$ and consider the case $n=0$. Note that $\II_0$ includes negative times, but this does not impose any problem for the proof. We take the case $n=0$ so that $\II_n^{(2)}$ is a union of $\I_1,\ldots,\I_{C/\varepsilon}$, which avoids the notational complications of having to enumerate $\{k:\I_k\subset\II_n^{(2)}\}$. 

To prove Statement (I), we begin by pointing out that
$$(\text{\ref{A1}})^c ~=~ \bigcup_{k=1}^{C/\varepsilon}\bigcup_{x\in\Lambda_i}\{|\C_k^\ajar(x)|>\log^2 L\}.$$
Thus, the main objective is to provide an appropriate bound on $\Q(|\C_k^\ajar(x)|>\log^2 L)$. To that end, we have for each $k=1,\ldots,C/\varepsilon$ and $e\in E(\Lambda_i)$ the following inclusion:
$$\{\zeta_k(e)=1\} ~\subseteq~ \{\eta_{t_{k-1}}(e)=1\}\cup\{e~\text{refreshes during}~\I_k\}.$$
Since each $\I_k$ has length $\varepsilon/\lambda$,
$$\Q(e~\text{refreshes during}~\I_k) ~=~ 1-e^{-\varepsilon} ~\leq~ \varepsilon.$$
Moreover, for any $t^*\in[0,C)$, we have that
\begin{align*}
    \big\{\eta_\frac{t^*}{\lambda}(e)=1\big\} ~\subseteq~ \{e~\text{refreshes during}~[-C/\lambda,t^*/\lambda),~\text{last refresh to}~1\}\\
    \cup\{e~\text{doesn't refresh during}~[-C/\lambda,t^*/\lambda)\}
\end{align*}
and hence
\begin{align*}
    \Q\big(\eta_\frac{t^*}{\lambda}(e)=1\big) ~&\leq~ (1-e^{-(C+t^*)})p+e^{-(C+t^*)} \\
    &\leq~ (1-e^{-2C})p+e^{-C},
\end{align*}
which yields that $\zeta_k$ is stochastically dominated by a product of $\mathrm{Ber}(\tilde{p})$ measures, where $\tilde{p}:=(1-e^{-2C})p+e^{-C}+\varepsilon$, noting that $\tilde{p}$ approaches $p+\varepsilon$ as $C\to\infty$. Thus, by choosing $L$ sufficiently large, we can achieve that $\varepsilon<\frac{1}{2}(p_c-p)$, so there exists $C_*$ so that for any $C\geq C_*$, the law of each $\zeta_k$ is stochastically dominated by Bernoulli percolation with density $\frac{1}{2}(p_c+p)$. It follows (see for example Theorem 5.6 in \cite{Georgii_Haggstrom_Maes_1999}) that there exists $D=D(d,p)$, so that
$$\Q(|\C_k^\ajar(x)|>\log^2 L) ~\leq~ e^{-D\log^2 L},$$
since the event $\{|\C(x)|>\log^2 L\}$ is increasing. Taking a union bound, we obtain that
\begin{align*}
     \Q\big((\mathrm{\ref{A1}})^c\big) ~&\leq~ \frac{C}{\varepsilon}L^de^{-D\log^2 L} \\
     &=~ e^{\log C+5\log\log L+d\log L-D\log^2 L}.
\end{align*}
It is thus sufficient to argue that there exists $c>0$, such that for $L$ large enough, 
$$\log C+5\log\log L+d\log L-D\log^2 L ~\leq~ -\log 3-c\log L^2;$$
writing $z:=\log L$, this is equivalent to requiring that for large enough $z$ we have
$$(D-c)z^2-dz-5\log z-\log(3C) ~\geq~ 0,$$
which indeed holds as long as $c<D$, noting that $D$ only depends on $d$ and $p$.

Before going on to prove Statement (II) and Statement (III), we introduce some tools that will be helpful in doing this. First we define the \textit{disagreement function} between configurations. Writing $\widetilde\Lambda_i$ for the $L/100$-thinning of $\Lambda_i$ (that is, $\Lambda_i$ is a $(L/100)$-thickening of $\widetilde\Lambda_i$), given $\Delta\subset\widetilde\Lambda_i$, we define a function $\D_\Delta:S^{\Lambda_i}\times S^{\Lambda_i}\to[0,\infty)$ as 
$$\D_\Delta(\sigma,\sigma') ~:=~ \sum_{x\in\widetilde\Lambda_i}\alpha^{\dist(x,\Delta)}\1_{\{\sigma(x)\neq\sigma'(x)\}}, \quad \sigma,\sigma'\in S^{\Lambda_i},$$
where $\alpha=\alpha_L:=e^{-1/\log^2 L}$. It is immediate that $\D_\Delta(\sigma,\sigma')<1$ implies that $\sigma(x)=\sigma'(x)$ for all $x\in\Delta$. Moreover, we define the stopping time
$$\tau ~:=~ \min\Big\{k\in\N:\max_{x\in\Lambda_i}|\C_k^\ajar(x)|>\log ^2 L\Big\},$$
noting that $\tau>C/\varepsilon$ implies that Condition (\ref{A1}) holds. To prove Statement (II) and Statement (III), we will utilize the following claim, the proof of which is provided after the end of this proof. 
\begin{claim}\label{almost_martingale_bound}
    Fix $\Delta\subset\widetilde\Lambda_i$, let $\varepsilon:=\log^{-5} L$ and define a stochastic process $(Y_k)_{k=1}^{C/\varepsilon}$ via
    $$Y_k ~:=~ \1_{\{\tau>k\}}\D_\Delta(\sigma_{t_k},\sigma_{t_k}').$$
    Moreover, let $\F_k$ denote the $\sigma$-algebra generated by $(V,T,U)\in\upd_{i,0}^{(2)}(\sigma)\cup\upd_{i,0}^{(2)}(\sigma')$ and $(E,T,U)\in\upd_{i,0}(\eta)$ with $T\leq t_k$. Then, for an appropriate choice of $L=L_\lambda$ satisfying \eqref{eq:L_assump}, writing $\EE$ of expectation under $\Q$,
    \begin{equation}\label{eq:claim_bound}
    \EE[Y_{k+1}|\F_k] ~\leq~ (3de\log^{-1}L) Y_k+2dL^{d-1}\alpha^{\dist(\partial\widetilde{\Lambda}_i,\Delta)}, \quad k=0,\ldots,\frac{C}{\varepsilon}-1,
    \end{equation}
    provided $\lambda$ is sufficiently small.
\end{claim}

To prove Statement (II), we first define a stochastic process
    $$N_k ~:=~ \1_{\{\tau>k\}}\D_{\overline{\Lambda_i^\core}}(\sigma_{t_k},\sigma_{t_k}'), \quad k=0,\ldots,\frac{C}{\varepsilon}.$$
    Claim \ref{almost_martingale_bound} tells us that for each $k=1,\ldots,C/\varepsilon$, 
    $$\EE[N_k|\F_{k-1}] ~\leq~ \gamma_L N_{k-1}+\widetilde{C}_{L,\overline{\Lambda_i^\core}},$$
    where $\gamma_L:=3de\log^{-1}L$ and $\widetilde{C}_{L,\overline{\Lambda_i^\core}}:=2dL^{d-1}\alpha^{\dist(\partial\widetilde{\Lambda}_i,\overline{\Lambda_i^\core})}$. Using that $\dist(\partial\widetilde{\Lambda}_i,\overline{\Lambda_i^\core})=L/10-2\cdot(L/100)=8L/100$, we see that
    $$\widetilde{C}_{L,\overline{\Lambda_i^\core}} ~=~ 2dL^{d-1}\alpha^\frac{8L}{100} ~=~ 2de^{(d-1)\log L-\frac{8L}{100\log^2 L}}.$$
     Noticing that $N_\frac{C}{\varepsilon}<1$ implies that either $\D_{\overline{\Lambda_i^\core}}(\sigma_\frac{(n+1)C}{\lambda},\sigma_\frac{(n+1)C}{\lambda}')<1$ (which implies that $\sigma_\frac{(n+1)C}{\lambda}$ and $\sigma_\frac{(n+1)C}{\lambda}'$ agree on $\overline{\Lambda_i^\core}$, i.e., Condition (\ref{A2}) holds) or that $\tau<k$ (which implies that Condition (\ref{A1}) fails), we find that
    $$\Q\big((\text{\ref{A1}})\cap(\text{\ref{A2}})^c\big) ~\leq~ \Q(N_\frac{C}{\varepsilon}\geq 1) ~\leq~ \EE[N_\frac{C}{\varepsilon}],$$
    where the second inequality is a simple application of Markov inequality. Using the relation $\EE[N_{k+1}|\F_k]\leq\gamma_L N_k+\widetilde{C}_{L,\overline{\Lambda_i^\core}}$, which we obtained above, we can further see that 
    $$\EE[N_k] ~\leq~ \gamma_L^k\EE[N_0]+\sum_{j=0}^{k-1}\gamma_L^j\widetilde{C}_{L,\overline{\Lambda_i^\core}} ~\leq~ \gamma_L^k L^d+\widetilde{C}_{L,\overline{\Lambda_i^\core}}\frac{1-\gamma_L^k}{1-\gamma_L}.$$
    Using that $\frac{1-\gamma_L^k}{1-\gamma_L}\leq\frac{1}{1-\gamma_L}$, we can assume that $L$ is sufficiently large, so that $\frac{1}{1-\gamma_L}\leq2$.\footnote{The choice of $2$ here is completely arbitrary, any fixed real number strictly larger $1$ works.} This yields, that
    $$\tilde{\EE}[N_\frac{C}{\varepsilon}] ~\leq~ \gamma_L^\frac{C}{\varepsilon}L^d+2\widetilde{C}_{L,\overline{\Lambda_i^\core}},$$
    and hence
    $$\Q\big((\text{\ref{A1}})\cap(\text{\ref{A2}})^c\big) ~\leq~ e^{C\log^5 L(\log(3de)-\log\log L)+d\log L}+e^{\log(4d)+(d-1)\log L-\frac{8L}{100\log^2 L}}.$$
    It is sufficient to argue that for large $L$,
   \begin{align*}
        C\log^5 L(\log(3de)-\log\log L)+d\log L ~&\leq~ -\log 6-c\log^2 L, \\
        \log(4d)+(d-1)\log L-\frac{8L}{100\log^2 L} ~&\leq~ -\log 6-c\log^2 L, 
   \end{align*}
    or equivalenlty, writing $z=\log L$, that for large $z$,
    \begin{align*}
        Cz^5\log z-C\log(3de)z^5-cz^2-dz-\log 6 ~&\geq~ 0, \\
        \frac{2}{25}\frac{e^z}{z^2}-cz^2 -(d-1)z-\log(24d) ~&\geq~ 0,
    \end{align*}
    which indeed holds.

    Lastly, to prove Statement (III), we define a stochastic process
    \begin{align*}
    M_k ~&:=~ \1_{\{\tau>k\}}\D_{\Lambda_i^\core}\big(\sigma_{t_k},\sigma_{t_k}'\big), \quad k=0,\ldots,\frac{C}{\varepsilon}.
    \end{align*}
    Claim \ref{almost_martingale_bound} again tells us that 
    $$\EE[M_{k+1}|\F_k]~\leq~\gamma_L M_k+\widetilde{C}_{L,\Lambda_i^\core},$$ where $\gamma_L$ is as above and $\widetilde{C}_{L,\Lambda_i^\core}:=2dL^{d-1}\alpha^{\dist(\partial\widetilde{\Lambda}_i,\overline{\Lambda_i^\core})}$. Since $\Lambda_i^\core\subset\overline{\Lambda_i^\core}$ (and hence $\dist(\partial\widetilde{\Lambda}_i,\Lambda_i^\core)\geq\dist(\partial\widetilde{\Lambda}_i,\overline{\Lambda_i^\core})$) it is immediate that $\widetilde{C}=\widetilde{C}_{L,\Lambda_i^\core}\to 0$ as $L\to\infty$.
    Note that by definition, $M_k\leq N_k$. Moreover, having imposed the assumption that $\sigma_{t_0}$ and $\sigma_{t_0}'$ agree on the entirety of $\overline{\Lambda_i^\core}$ additionally tells us that
    $$M_0 ~\leq~ \alpha^\frac{L}{100}N_0 ~<~ \alpha^\frac{L}{100}L^d ~=:~ \alpha',$$
    which follows immediately from the fact that for each $x\in\widetilde{\Lambda}_i\setminus\overline{\Lambda_i^\core}$, $\dist(x,\Lambda_i^\core)=\dist(x,\overline{\Lambda_i^\core})+L/100$. Recalling that $\alpha=e^{-1/\log^2 L}$, it is clear that this bound tells us that for sufficiently large $L$, $M_0\ll 1$; since we are taking $\lambda$ small (and hence $L$ large) we may assume that $L$ is such that $\sqrt{\alpha'}< 1/e=\alpha^{\log^2 L}$. It is immediate, by definition of the process, that if $M_1,\ldots,M_\frac{C}{\varepsilon}<\sqrt{\alpha'}$, either (i) $\sigma$ and $\sigma'$ agree on the entirety of the $(\log^2 L)$-thickening of $\Lambda_i^\core$ at times $t_1,\ldots,t_{C/\varepsilon}$, or (ii) Condition (\ref{A1}) fails, i.e., $\tau<C/\varepsilon$. In other words, writing $\A$ for the event in scenario (i), we have that
    $$\{M_1,\ldots,M_\frac{C}{\varepsilon}<\sqrt{\alpha'}\} ~\subseteq~ (\text{\ref{A1}})^c\cup\A.$$
    Note also, that if $\sigma_{t_{k-1}}$ and $\sigma_{t_{k-1}}'$ agree on $(\log^2 L)$-thickening of $\Lambda_i^\core$, and the same holds for $\sigma_{t_k}$ and $\sigma_{t_k}'$, then $\sigma_t$ and $\sigma_t'$ agree on $\Lambda_i^\core$ for all $t\in(t_{k-1},t_{k})$, as soon as all $\I_{k}$-ajar clusters are of size at most $\log^2 L$ (which tells us that disagreement from outside $(\log^2 L)$-thickening of $\Lambda_i^\core$ cannot enter $\Lambda_i^\core$). It follows that
    $$(\text{\ref{A1}})\cap\A ~\subseteq~ (\text{\ref{A1}})\cap(\text{\ref{A3}}).$$
    We can now write 
    \begin{align*}
        (\text{\ref{A1}})\cap(\text{\ref{A3}})^c  ~&=~ \Big((\text{\ref{A1}})\cap\A^c\cap(\text{\ref{A3}})^c\Big)\cup\Big((\text{\ref{A1}})\cap\A\cap(\text{\ref{A3}})^c\Big) \\
        &\subseteq~ \Big((\text{\ref{A1}})\cap\A^c\Big)\cup\Big((\text{\ref{A1}})\cap\underbrace{(\text{\ref{A3}})\cap(\text{\ref{A3}})^c}_{\emptyset}\Big) \\
        &\subseteq~ \{M_1,\ldots,M_\frac{C}{\varepsilon}<\sqrt{\alpha'}\}^c \\
        &=~ \bigcup_{k=1}^{C/\varepsilon}\{M_k\geq\sqrt{\alpha'}\},
    \end{align*}
    from which it follows that 
    $$\Q\big((\text{\ref{A1}})\cap(\text{\ref{A3}})^c\big) ~\leq~ \sum_{k=1}^{C/\varepsilon}\Q(M_k\geq\sqrt{\alpha'}) ~\leq~ \frac{1}{\sqrt{\alpha'}}\sum_{k=1}^{C/\varepsilon}\EE[M_k].$$
    Recalling that $\EE[M_{k+1}|\F_k]\leq\gamma\EE[M_k]+\widetilde{C}_{\Lambda_i^\core}$ (suppressing the dependence of the constants on $L$), it follows that 
    \begin{align*}
        \EE[M_k] ~&\leq~ \gamma^k\EE[M_0]+\frac{1-\gamma^k}{1-\gamma}\widetilde{C}_{\Lambda_i^\core} \\
        &\leq~ \gamma^k\alpha' + \frac{1-\gamma^k}{1-\gamma}2dL^{d-1}\alpha^{\dist(\partial\widetilde{\Lambda}_i,\Lambda_i^\core)} \\
        &\leq~ \gamma^k\alpha'+\frac{1-\gamma^k}{1-\gamma}\frac{2d}{L}\alpha^{\dist(\partial\widetilde{\Lambda}_i,\Lambda_i^\core)}\alpha' \\
        &<~ \alpha'
    \end{align*}
    for all $L$ large enough and hence finally
    \begin{align*}
    \Q\big((\text{\ref{A1}})\cap(\text{\ref{A3}})^c\big) ~&\leq~ \frac{C}{\varepsilon}\sqrt{\alpha'} \\
    &=~ e^{\log C+5\log\log L+\frac{1}{2}(d\log L-\frac{L}{100\log^2 L})}.
    \end{align*}
    It is sufficient to argue that for large enough $L$, we have
    $$\log C+5\log\log L+\frac{d}{2}\log L-\frac{L}{200\log^2 L} ~\leq~ -\log 3-c\log^2 L,$$
    or equivalently, writing $z=\log L$, that for $z$ large,
    $$\frac{1}{200}\frac{e^z}{z}-cz^2-\frac{d}{2}z-5\log z-\log(3C) ~\geq~ 0,$$
    which indeed holds.
\end{proof}

Before proving Claim \ref{almost_martingale_bound}, we briefly provide another claim that we use will use in the proof. Let $F$ be a connected subgraph of $\T_N^d$, and let $\PP_{\opt,F}^T$ denote the $T$-optimal coupling for $\Phi^{(N)}$-Glauber dynamics on $F$. Letting
$$\delta_F^{(N)}(T) ~:=~ \max_{\sigma_0,\sigma_0'}\PP_{\opt,F}^T(\sigma_T\neq\sigma_T'),$$
where $(\sigma_t)_{t\geq 0},(\sigma_t')_{t\geq 0}$ are two copies of the process, started from $\sigma_0,\sigma_0'\in S^{V(F)}$, we define
$$\delta(a,T) ~:=~ \sup_{N\geq 1}\max_{F:|V(F)|\leq\log^2 a}\delta_F^{(N)}(T),$$
where the maximum is over connected subgraphs of $\T_N^d$. Given this definition, on any cluster of size at most $\log^2 L$, where $\varepsilon/\lambda$-optimal coupling is ran during an interval $\I_k$, the copies of the process will agree at the end of $\I_k$ with probability at least $1-\delta(L,\varepsilon/\lambda)$.

\begin{claim}\label{cl:delta_to_0}
    Let $a\in\N$. Then, for any $c\in(0,1)$,
    $$\delta(a,T) ~\leq~ c, \quad \forall T\geq h_{c/2}(\log^2 a),$$
    where $h_{c/2}(\cdot)$ is defined as in \eqref{eq:def_h}.
\end{claim}

\begin{proof}
    Recall first that given $(\sigma_t)_{t\geq 0},(\sigma_t')_{t\geq 0}$ defined on $F$, the probability $\PP_{\opt,F}^T(\sigma_T\neq\sigma_T')$ equals the total variation between the laws of $\sigma_T$ and $\sigma_T'.$ Noting that for a chain with stationary measure $\pi$,
    $$\max_{x,y}\|\Prm_x^t-\Prm_y^t\|_\TV ~\leq~ 2\max_x\|\Prm_x^t-\pi\|_\TV,$$
    it follows from the definition of mixing time that
    $$\delta_F^{(N)}(T)~<~c, \quad \forall T\geq \tilde t_\mix^{(N)}(c/2,F).$$
    Thus, the claim follows from the definition of $h_{c/2}(\cdot).$
\end{proof}

\begin{proof}[Proof of Claim \ref{almost_martingale_bound}]
    It is straightforward that
    $$Y_k ~=~ \1_{\{\tau>k\}}\sum_{\C}\underbrace{\sum_{x\in V(\C)\cap\widetilde\Lambda_i}\alpha^{\dist(x,\Delta)}\1_{\{\sigma_{t_k}(x)\neq\sigma_{t_k}'(x)\}}}_{\Psi_k(\C)},$$
    where the first sum is over all $\eta_{t_{k}}$-clusters that intersect $\widetilde\Lambda_i$. For convenience, write $\C_{x,k}:=\C_{t_k}(x)$ and $\overline{\C_{x,k}}:=\C_k^\ajar(x)$. Recall that if at time $t_{k-1}$, $\sigma$ and $\sigma'$ agree on $\overline{\C_{x,k}}$, then identity coupling runs on it during $\I_k$ (noting that $\overline{\C_{x,k}}$ is a union of several $\eta_{t_{k-1}}$-clusters), so agreement is preserved and hence $\Psi_{x,k}:=\Psi_k(\overline{\C_{x,k}})=0$. Since each $\I_k$-ajar cluster can also be decomposed into a union of several $\eta_{t_k}$-clusters, we can replace the sum $\sum_{\C}\Psi_k(\C)$ with $\sum_{\overline\C}\Psi_k(\overline\C)$, where the last sum is over all $\I_k$-ajar clusters. Using that for disjoint $\C_1,\C_2$ we have $\Psi_k(\C_1\cup\C_2)=\Psi_k(\C_1)+\Psi_k(\C_2)$, it follows that
    \begin{align*}
        Y_k ~&\leq~ \1_{\{\tau>k\}}\sum_{\overline{\C}}\1_{\{\sigma_{t_{k-1}}(\overline{\C})\neq\sigma_{t_{k-1}}'(\overline{\C})\}}\Psi_k(\overline{\C}) \\
        &\leq~ \1_{\{\tau>k\}}\sum_{\substack{x\in\widetilde{\Lambda}_i:\\\sigma_{t_{k-1}}(x)\neq\sigma_{t_{k-1}}'(x)}}\Psi_{x,k} + \underbrace{\sum_{x\in\partial\widetilde{\Lambda}_i}\alpha^{\dist(x,\Delta)}}_{\widetilde{C}},
    \end{align*}
    where the first sum bounds from above the sum over all $\overline{\C}$ which are contained in $\widetilde{\Lambda}_i$, while in the second we take care of the possibility of clusters not fully contained in $\overline{\C}$ only having disagreement (at time $t_{k-1}$) outside $\widetilde{\Lambda}_i$. Clearly,
    $$\EE[Y_{k+1}|\F_k] ~\leq~ \widetilde{C} + \sum_{\substack{x\in\widetilde{\Lambda}_i:\\\sigma_{t_k}(x)\neq\sigma_{t_k}'(x)}}\EE[\1_{\{\tau>k+1\}}\Psi_{x,k+1}|\F_k],$$
    so it remains to bound the conditional expectation inside the sum. To this end, define the following two families of random variables:
    \begin{itemize}
        \item Let $U_{x,k}$ be the indicator random variable for the event that there is an edge in $E^+(\C_{x,k-1})$ which updates during $\I_k$. Note that $U_{x,k}=0$ implies that the optimal coupling is ran on $\C_{x,k-1}$, unless there was already agreement between $\sigma$ and $\sigma'$ on this cluster; moreover, it tells us that $\C_{x,k-1}=\overline{\C_{x,k}}=\C_{x,k}$.
        \item If $U_{x,k}=0$, we can consider the probability that the optimal coupling succeeds: we let $S_{x,k}$ be a coin flip associated with this event, where $S_{x,k}=1$ corresponds to coupling succeeding.
    \end{itemize}
    We can rewrite the conditional expectation above as
    \begin{align*}
        \EE[\1_{\{\tau>k+1\}}\Psi_{x,k+1}|\F_k] ~&=~ \EE[\1_{\{\tau>k+1\}}\Psi_{x,k+1}U_{x,k+1}|\F_k] \\
        &\quad+ \EE[\1_{\{\tau>k+1\}}\Psi_{x,k+1}(1-U_{x,k+1})S_{x,k+1}|\F_k] \\
        &\quad+ \EE[\1_{\{\tau>k+1\}}\Psi_{x,k+1}(1-U_{x,k+1})(1-S_{x,k+1})|\F_k].
    \end{align*}
    We first treat the second term, which is trivial, since $(1-U_{x,k+1})S_{x,k+1}$ is $1$ if there are no updates to edges in $E^+(\C_{x,k})$ during $\I_{k+1}$ and that the optimal coupling succeeds, so it follows immediately that
    $$\Psi_{x,k+1}(1-U_{x,k+1})S_{x,k+1} ~=~ 0,$$
    so the whole term equals $0$. 

    Next, we consider the third term. Since $(1-U_{x,k+1})(1-S_{x,k+1})$ is the indicator on the event that there are no updates to edges in $E^+(\overline{\C_{x,k}})$ during $\I_{k+1}$ and that the optimal coupling fails, we get that 
    \begin{align*}
        \Psi_{x,k+1}(1-U_{x,k+1})(1-S_{x,k+1}) ~&\leq~ (1-U_{x,k+1})(1-S_{x,k+1})\sum_{y\in\C_{x,k}}\alpha^{\dist(x,\Delta)}.
    \end{align*}
    Moreover, since on $\{\tau>k+1\}$, all clusters are of size at most $\log^2 L$ during $\I_k$, we have that
    $$\1_{\{\tau>k+1\}}\sum_{y\in\C_{x,k}}\alpha^{\dist(y,\Delta)} ~\leq~ \alpha^{\dist(x,\Delta)-\log^2 L}\log^2 L,$$
    which yields that the third term is bounded from above by
    $$(\alpha^{\dist(x,\Delta)-\log^2 L}\log^2 L)\EE[\1_{\{\tau>k+1\}}(1-U_{x,k+1})(1-S_{x,k+1})|\F_k].$$
    Letting $\F_{k}^\eta$ denote the $\sigma$-algebra generated by $\{(E,T,U)\in\upd_{i,0}(\eta):T\leq t_k\}$, we can use the tower property to obtain
    \begin{align*}
        &\EE[\1_{\{\tau>k+1\}}(1-U_{x,k+1})(1-S_{x,k+1})|\F_k] \\
        &=~ \EE[\1_{\{\tau>k+1\}}(1-U_{x,k+1})\EE[(1-S_{x,k+1})|\F_k\vee\F_{k+1}^\eta]|\F_k],
    \end{align*}
    where $\F_k\vee\F_{k+1}^\eta$ denotes the $\sigma$-algebra generated by $\F_k\cup\F_{k+1}^\eta$. Writing, as in Claim \ref{cl:delta_to_0}, $\delta:=\delta(L,\varepsilon/\lambda)$ for the upper bound on the probability that the optimal coupling fails, we get that the expressions in the equation above are bounded from above by
    $$\delta\EE[\1_{\{\tau>k+1\}}(1-U_{x,k+1})|\F_k] ~\leq~ \delta\EE[\1_{\{\tau>k\}}|\F_k] ~\leq~ \1_{\{\tau>k\}}\delta,$$
    where we use that $\{\tau>k+1\}\subset\{\tau>k\}\in\F_k$. To sum up, the third term is bounded from above by
    $$\1_{\{\tau>k\}}\delta\alpha^{\dist(x,\Delta)-\log^2 L}\log^2 L.$$
    Last, we consider the first term. Since $U_{x,k+1}$ is $1$ if there was an update to some edge in $E^+(\C_{x,k})$ during $\I_{k+1}$, we simply do
    \begin{align*}
        \1_{\{\tau>k+1\}}\Psi_{x,k+1}U_{x,k+1} ~&\leq~ \1_{\{\tau>k+1\}}U_{x,k+1}\sum_{y\in\overline{\C_{x,k}}}\alpha^{\dist(y,\Delta)} \\
        &\leq~ \1_{\{\tau>k+1\}}U_{x,k+1}\alpha^{\dist(x,\Delta)-\log^2 L}\log^2 L.
    \end{align*}
    Noting that a cluster of size at most $\log^2 L$ has less than $2d\log^2 L$ adjacent edges, we see that the probability of an edge update occurring (i.e., $U_{x,k+1}=1$) is at most $1-\exp(-\varepsilon2d\log^2 L)<\varepsilon2d\log^2 L=2d\log^{-3}L$,
    which yields that 
    $$\EE[\1_{\{\tau>k+1\}}U_{x,k+1}|\F_k] ~\leq~ 2d\log^{-3}\!L\,\EE[\1_{\{\tau>k+1\}}|\F_k] ~\leq~ \1_{\{\tau>k\}}2d\log^{-3}L,$$
    so the whole first term is bounded from above by
    $$\1_{\{\tau>k\}}2d\log^{-3}L\alpha^{\dist(x,\Delta)-\log^2 L}\log^2 L.$$
    Summing everything up, we get that
    $$\EE[\1_{\{\tau>k+1\}}\Psi_{x,k+1}|\F_k] ~\leq~ \1_{\{\tau>k\}}\alpha^{\dist(x,\Delta)-\log^2 L}\log^2 L(\delta+2d\log^{-3}L),$$
    and hence it follows that
    \begin{align*}
    \EE[Y_{k+1}|\F_k] ~&\leq~ \widetilde{C}+\1_{\{\tau>k\}}\alpha^{-\log^2 L}\log^2 L(\delta+2d\log^{-3}L)\sum_{\substack{x\in\widetilde{\Lambda}_i:\\\sigma_{t_k}(x)\neq\sigma_{t_k}'(x)}}\alpha^{\dist(x,\Delta)} \\
    &\leq~ 2dL^{d-1}\alpha^{\dist(\partial\widetilde{\Lambda}_i,\Delta)} + \alpha^{-\log^2 L}\log^2 L(\delta+2d\log^{-3}L)Y_k.
    \end{align*}
     In order to obtain \eqref{eq:claim_bound}, it remains to show that for $\lambda$ small enough, $\delta=\delta(L,\varepsilon/\lambda)\leq d\log^{-3}L$.
    Following Claim \ref{cl:delta_to_0}, it is sufficient to show that there exists $\lambda_*$, such that for $\lambda\leq\lambda_*$,
    $$\frac{\log^{-5}L}{\lambda} ~\geq~ h_{\frac{d}{2}\log^{-3}L}(\log^2 L).$$
    To this end, we note that we can always pick some increasing function $\psi$, so that $h_{(d/2)\log^{-3}L}(\log^2 L)\log^{5}L\leq\psi(L)$. Under the assumptions of \eqref{eq:L_assump}, it is possible to take $L=L_\lambda$ to be a sufficiently slowly growing function of $1/\lambda$, so that $\psi(L)\leq1/\lambda$ for small enough $\lambda$.
\end{proof}


\subsection{Sufficient condition for agreement}\label{subsec:sufficient_perc}

As explained above, the event that $\B_{i,n}$ is \textit{good} gives us several important bits of information, the first being that at $t=(n+1)C/\lambda$, $\sigma_t$ and $\sigma_t'$ agree on the entirety of $\overline{\Lambda_i^\core}\supset\Lambda_i^\core$. Following this, one could simply wait for the first $n$, such that $\B_{i,n}$ is good for all $i\in\T_M^d$. This however, regardless of $\Q(\B_{i,n}~\text{is good})$ being arbitrarily high, takes too long to happen. 

The other important information is obtained from Condition (\ref{A3}). In particular, if the preceding box, $\B_{i,n-1}$, is also good, then we obtain that $\sigma_t$ and $\sigma_t'$ agree on the entirety of $\Lambda_i^\core$ for all $t\in\II_n^{(2)}=[nC/\lambda,(n+1)C/\lambda)$. We will exploit this to obtain a sufficient condition for agreement of processes at times of form $nC/\lambda$.

In what follows we will write, for a given $n\geq 1$, we will write $J^{(n)}=\{j\in\T_M^d:\B_{j,n}~\text{is good}\}$ and denote by $A_1^{(n)},\ldots,A_{\ell_n}^{(n)}$ the connected components of bad boxes in the $\ell^\infty$-distance, i.e., the coarsest partition of $\T_M^d\setminus J^{(n)}$ such that for any $i\in A_k^{(n)}$ and $j\in A_{k'}^{(n)}$ with $k\neq k'$ we have $\|i-j\|_\infty>1$. Moreover, for $k\in\{1,\ldots,\ell_n\}$, we will write $J_k^{(n)}$ for all $j\in J^{(n)}$ for which there exists $i\in A_k^{(n)}$ such that $\|i-j\|_\infty=1$.

\begin{lemma}\label{lemma:protection}
Fix $n\geq 2$ and and let $k\in\{1,\ldots,\ell_n\}$ be such that
$$\sigma_\frac{nC}{\lambda}(x) ~=~ \sigma_\frac{nC}{\lambda}'(x), \quad \forall x\in\bigcup_{j\in A_k^{(n)}\cup J_k^{(n)}}\overline{\Lambda_j^\core}.$$
Then, for all $t\in\II_n^{(2)}$ and $x\in\bigcup_{j\in A_k^{(n)}}\Lambda_j^\core$,
\begin{equation}\label{eq:agr}
    \sigma_t(x) ~=~ \sigma_t'(x).
    \end{equation}
\end{lemma}

\begin{proof}
    Consider first the case where $\ell_n=1$, i.e., $A_1^{(n)}=\T_M^d\setminus J^{(n)}$. In this case, $\sigma_\frac{nC}{\lambda}(x)=\sigma_\frac{nC}{\lambda}'(x)$ for all $x\in\T_N^d$, so it follows from the definition of the coupling that $\sigma_t=\sigma_t'$ for all $t\geq nC/\lambda$. We assume from now on that $\ell_n\geq 2$. Let now $k$ be such that $\sigma_\frac{nC}{\lambda}(x)=\sigma_\frac{nC}{\lambda}'(x)$ for all $x\in\bigcup_{j\in A_k^{(n)}\cup J_k^{(n)}}\overline{\Lambda_j^\core}$. By definition of $J$ and the assumption that $\sigma_\frac{nC}{\lambda}$ and $\sigma_\frac{nC}{\lambda}'$ agree on $\bigcup_{j\in J_k^{(n)}}\overline{\Lambda_j^\core}$, it follows that
    \begin{equation}\label{eq:conseq_of_J}
        \sigma_t(\Lambda_j^\core) ~=~ \sigma_t'(\Lambda_j^\core), \quad \forall t\in\II_n^{(2)},j\in J_k^{(n)}.
    \end{equation}
    For $i\in A_k^{(n)}$, we write $\tau_i=\inf\{t\in\II_n^{(2)}:\sigma_t(\Lambda_i^\core)\neq\sigma_t'(\Lambda_i^\core)\}$ and $\tau=\min_{i\in A_k^{(n)}}\tau_i$. We now argue by contradiction, assuming $\tau<\infty$. We let $m$ be such that $\tau\in\I_{m+1}$, so that
    $$\sigma_{t_m}(\Lambda_i^\core) ~=~ \sigma_{t_m}'(\Lambda_i^\core), \quad \forall i\in A_k^{(n)}.$$
    It follows that all sites in $\bigcup_{i\in A_k^{(n)}}\Lambda_i^\core$ belong to the region where identity coupling is run during $\I_{m+1}$. Writing $x^*$ for the unique site in $\bigcup_{i\in A_k^{(n)}}\Lambda_i^\core$ such that $\sigma_\tau(x^*)\neq\sigma_\tau'(x^*)$, it then follows by graphical representation that there must exist $y^*$ and $\varepsilon'>0$ such that $x^*y^*\in E(\T_N^d)$ and $\sigma_s(y^*)\neq\sigma_s'(y^*)$ for $s\in[\tau-\varepsilon',\tau]$. By definition of clusters $A_1^{(n)},\ldots,A_{\ell_n}^{(n)}$, such $y^*$ must belong to either $\bigcup_{i\in A_k^{(n)}}\Lambda_i^\core$ or $\bigcup_{j\in J_k^{(n)}}\Lambda_j^\core$. However, belonging to the former would contradict the definition of $\tau$ and beloning to the latter would contradict \eqref{eq:conseq_of_J}, so we conclude that $\tau=\infty$ and hence the result follows.
\end{proof}

We can state and prove a sufficient condition for agreement.

\begin{lemma}\label{lemma:sufficient_condition_paths}
    In order to reach 
    \begin{equation}
        \sigma_\frac{(n+1)C}{\lambda}(x) ~=~ \sigma_\frac{(n+1)C}{\lambda}'(x), \quad \forall x\in\T_N^d,
    \end{equation}
    for $n\geq 2$, it is sufficient that there exist no $\ell^\infty$-paths of bad boxes from $\T_M^d\times\{1\}$ to $\T_M^d\times\{n\}$. 
\end{lemma}

\begin{proof}
    Let $x\in\T_N^d$ be such that $\sigma_\frac{(n+1)C}{\lambda}(x)\neq \sigma_\frac{(n+1)C}{\lambda}'(x)$. Letting $i\in\T_M^d$ be such that $x\in\Lambda_i^\core$, it follows that $\B_{i,n}$ is a bad box; let $k\in\{1,\ldots,\ell_n\}$ be such that $i\in A_k^{(n)}$. By Lemma \ref{lemma:protection}, there exists $j\in A_k^{(n)}\cup J_k^{(n)}$ such that
    $$\sigma_\frac{nC}{\lambda}(\overline{\Lambda_j^\core}) ~\neq~ \sigma_\frac{nC}{\lambda}'(\overline{\Lambda_j^\core}),$$
    so in particular, $\B_{j,n-1}$ is a bad box. By definition of $A_k^{(n)}$ and $J_k^{(n)}$, there exists $j^*\in A_k^{(n)}$ so that $\|(j,n-1)-(j^*,n)\|_\infty=1$, from which it follows that there exists a path of $\ell^\infty$-path of bad boxes from $(i,n)$ to $(j,n-1)$. We can repeat this argument inductively until we reach a bad box of form $\B_{j^{**},1}$, which yields an $\ell^\infty$-path of bad boxes from $\T_M^d\times\{1\}$ to $(i,n)$. Thus, if there is no $\ell^\infty$-path of bad boxes from $\T_M^d\times\{1\}$ to $\T_M^d\times\{n\}$, configurations $\sigma_{\frac{(n+1)C}{\lambda}}$ and $\sigma_\frac{(n+1)C}{\lambda}$ cannot disagree on any site in $\T_N^d$.
\end{proof}

Define now, for each $i\in\T_M^d$ and $n\in\N$, a random variable
$$X_{i,n} ~:=~ \1_{\{\B_{i,n}~\text{is good}\}},$$
and write $\Prm_{\!\mathrm{good}}$ for the law of $(X_{i,n}:i\in\T_M^d,n\in\N)$ under $\Q$. Since the event $\{\B_{i,n}~\text{is good}\}$ depends only on updates within $\B_{i,n}$, it is independent of all the boxes which it does not intersect. Since $\Q(X_{i,n}=1)$ can be made arbitrarily high, it follows from the result of Liggett, Schonmann and Stacey (Theorem 0.0 in \cite{Liggett_Schonmann_Stacey_1997}) that there exists a $\rho=\rho(\Q(X_{i,n}=1))$, which too can be made arbitrarily high, such that
$$\P_{\rho}^\site ~\preceq_\D~ \Prm_{\!\mathrm{good}},$$
where $\P_\rho^\site$ denotes a Bernoulli site percolation on $\T_M^d\times\N$ with density $\rho$. Since the event that there exist no $\ell^\infty$-paths of $0$-sites from $\T_M^d\times\{1\}$ to $\T_M^d\times\{n\}$ is increasing, it is sufficient for us to give a lower bound on its probability w.r.t.~$\P_\rho^\mathrm{site}$, provided it is high enough to yield the result.

\begin{lemma}\label{lemma:site_percolation_bound}
    Let $\P_p^\site$ be a Bernoulli site percolation measure on $\T_M^d\times\N$ with density $p\in(0,1)$, and let $n\in\N$. Then, the probability of existence of an open $\ell^\infty$-path from $\T_M^d\times\{1\}$ to $\T_M^d\times\{n\}$ is bounded from above by $M^d(3^{d+1}-2)^n p^n$.
\end{lemma}

\begin{proof}
     Clearly, the event in question is contained in
     $$\{\text{$\exists$ open $\ell^\infty$-path of length $n$ started from $\T_M^d\times\{1\}$}\},$$
     whose probability is bounded from above by
     $$\sum_{x\in\T_M^d}\P_p^\site(\text{$\exists$ open $\ell^\infty$-path of length $n$ started from $(x,1)$}).$$
     Since each $v\in\T_M^d\times\N$ has at most $3^{d+1}-1$ neighbours in $\ell^\infty$-norm, and in particular, for $v'\in\T_M^d\times\{1\}$, there is $2\cdot 3^d-1$ of them, it follows that 
     $$\P_p^\site(\text{$\exists$ open $\ell^\infty$-path of length $n$ started from $(x,1)$}) ~\leq~ (3^{d+1}-2)^n p^n.$$
     Summing over all $x\in\T_M^d$ yields the bound.
\end{proof}

 We are finally equipped to give the proof of the main lemma in this section.

 \begin{proof}[Proof of Proposition~\ref{lemma:main_lemma_section_3}]
         Throughout the proof we keep in mind that $M$ is of order $N$, i.e., there exists a constant $c=c_\lambda\in(0,1)$, such that $M=cN$. We also note that by Lemma \ref{lemma:P(box_good)} and Theorem 0.0 in \cite{Liggett_Schonmann_Stacey_1997}, there exist for each $\rho\in(0,1)$ some $\lambda_\rho>0$ and $N_\rho\in\N$, such that $\P_\rho^\site\preceq_\D\Prm_{\!\mathrm{good}}$ as long as $\lambda<\lambda_\rho$ and $N>N_\rho$; throughout the proof we assume that $\lambda$ and $N$ are such that $\rho$ can be taken such that $(3^{d+1}-2)(1-\rho)<1$. Combining Lemma \ref{lemma:sufficient_condition_paths} and Lemma \ref{lemma:site_percolation_bound}, we obtain that
     \begin{align*}
         \Q\big(\sigma_\frac{(n+1)C}{\lambda}=\sigma_\frac{(n+1)C}{\lambda}'\big) ~&\geq~ \Q\big(\text{$\nexists$ $\ell^\infty$-path of $0$-sites from $\T_M^d\times\{1\}$ to $\T_M^d\times\{n\}$}\big) \\
         &\geq~ \P_\rho^\site\big(\text{$\nexists$ $\ell^\infty$-path of $0$-sites from $\T_M^d\times\{1\}$ to $\T_M^d\times\{n\}$}\big) \\
         &=~ 1-\P_{1-\rho}^\site\big(\text{$\exists$ $\ell^\infty$-path of $1$-sites from $\T_M^d\times\{1\}$ to $\T_M^d\times\{n\}$}\big) \\
         &\geq~ 1-M^d(3^{d+1}-2)^n(1-\rho)^n,
     \end{align*}
     and hence, for $\kappa\in(0,1/2)$,
     \begin{align*}
     \inf\{t\geq 0:\Q(\sigma_t\neq\sigma_t')\leq\kappa/2\} ~&\leq~ \frac{C}{\lambda}\big(1+\inf\{n\in\N:M^d(3^{d+1}-2)^n(1-\rho)^n\leq\kappa/2\}\big) \\
     &\leq~ \frac{2C}{\lambda}\left\lceil\frac{d\log N-\log(\kappa/2)}{|\log((3^{d+1}-2)(1-\rho))|}\right\rceil.
     \end{align*}
     Moreover, for $N$ large enough, the right hand side is bounded from above by
     $$\frac{\hat C_\rho}{\lambda}\log N, \quad \hat{C}_\rho:=\frac{4Cd}{|\log((3^{d+1}-2)(1-\rho))|}.$$
     Noting that the function $z\mapsto|\log((3^{d+1}-2)(1-z)|$ increases to $\infty$ as $z\uparrow 1$, it follows that for each $C^\star>0$ there exists $\rho$ sufficiently close to $1$, so that $\hat C_\rho\leq  C^\star$. In particular, if $\lambda<\lambda_\rho$ and $N>N_\rho$, where $\lambda_\rho$ and $N_\rho$ are as in the beginning of the proof, it holds that indeed
     $$\inf\{t\geq 0:\Q(\sigma_t\neq\sigma_t')\leq\kappa/2\}\leq \frac{C^\star}{\lambda}\log N,$$
     completing the proof.
\end{proof}


\section{Lower bound on the mixing time}\label{sec:lower_bound}

In the previous two chapters, we carried out the proof that for sufficiently small $\lambda>0$, we have $t_\mix=\Theta(\lambda^{-1}\log N)$, by first demonstrating that
$$t_\mix^\env(\kappa) ~\leq~ t_\mix(\kappa) ~\leq~ t_\mix^\env(\kappa/2) + \inf\{t\geq 0:\max_{f,g}\PP_{\!f,g}(\sigma_t\neq\sigma_t')\leq\kappa/2\},$$
where $\PP_{\!f,g}$ is any coupling as in Lemma \ref{lemma:upper_bound}. Then, we showed that $t_\mix^\env(\kappa)=\Theta(\lambda^{-1}\log N)$ and that for an appropriate choice of coupling $\PP_{\!f,g}$ we have 
$$\inf\{t\geq 0:\max_{f,g}\PP_{\!f,g}(\sigma_t\neq\sigma_t')\leq\kappa/2\}~=~O(\lambda^{-1}\log N).$$ However, since we only provided an upper bound for the quantity above, the question remains whether one can do better asymptotically. This is a particularly important question to address, as proving that it is also $o(\lambda^{-1}\log N)$ would imply a cutoff at the same location as the one for the random environment. 

Heuristically, however, this does not seem likely to be true. For the classical case (i.e., without the random environment), it was proven in \cite{Hayes_Sinclair_2007} that any nonredundant\footnote{This condition requires that under the dynamics, at least two different spin values can appear on any site.} Glauber dynamics on a sequence of finite undirected graphs of bounded degree has mixing time of order $\Omega(\log N)$. However, the proof in \cite{Hayes_Sinclair_2007} is restricted to reversible Markov chains. 

Below, we provide an analogue result in our setting for a large class of interactions, corresponding to monotone measures, which includes for example the Ising model. The proof is heavily inspired by the one in \cite{Hayes_Sinclair_2007}. 

Consider some linear order on $S$. We say that a measure $\mu$ on $S^{\T_N^d}$ is \textit{monotone} if the following holds for $X\sim\mu$ and any $\xi,\zeta\in S^{\T_N^d\setminus\{x\}}$ with $\xi\leq\zeta$ and $\mu(\xi),\mu(\zeta)>0$:
$$\mu\big(X(x)\geq s\big|X(\T_N^d\setminus\{x\})=\xi\big) ~\leq~ \mu\big(X(x)\geq s\big|X(\T_N^d\setminus\{x\})=\zeta\big),$$
for any $s\in S$ and $x\in\T_N^d$. A very useful notion that exists in this setting is that of a monotone coupling. In the context of Glauber dynamics, that translates to an ability to couple two copies $(\sigma_t)_{t\geq 0}$ and $(\sigma_t')_{t\geq 0}$ of the process with starting configurations $\sigma_0\leq\sigma_0'$ in such a way that $\sigma_t\leq\sigma_t'$ for all $t\geq 0$. We will moreover assume that the invariant measures corresponding to $\Phi^{(N)}$-spin systems on dynamical percolation satisfy the following non-triviality condition: 
\begin{equation}\label{eq:nonrtiv_condition}
\begin{aligned}
    &\text{there exists a partition $(S^+,S^-)$ of $S$ with $s_+>s_-$ for any $s_+\in S^+,s_-\in S^-$,} \\ &\text{such that there exists $\chi\in(0,1/2)$, so that} \\
    &\quad\quad\quad\quad\quad\quad\quad\quad\chi ~\leq~ \pi(\sigma(x)\in S^-) ~\leq~ 1-\chi\quad, \quad\forall x\in\T_N^d, \\
    &\text{uniformly in $N$.} 
\end{aligned}
\end{equation}

In what follows, we write $t_\mix^\sta(\cdot)$ for the mixing time for the corresponding spin system on dynamical percolation started from $\eta_0\sim\P_p$. 

\begin{theorem}\label{thm:lower_bound}
    Let $(\Phi^{(N)})_{N\geq 1}$ be as in Theorem \ref{thm}, and assume that their corresponding Gibbs measures are monotone and non-trivial as in \eqref{eq:nonrtiv_condition}. Then, for any $\kappa\in(0,1/2)$, there exists $N_0$, so that for all $N\geq N_0$,
    $$t_\mix^\sta(\kappa) ~\geq~ \frac{\log N}{30\log(2d)},$$
    regardless of the choice of $p$ and $\lambda$.
\end{theorem}

\begin{proof}
    It is sufficient to demonstrate that there exists a constant $D\geq (30\log(2d))^{-1}$ and an event $A$ such that, writing $T:=D\log N$,
    $$|\Prm_\sigma^T(A)-\pi(A)| ~>~ \kappa,$$
    for some $\sigma\in S^{\T_N^d}$. In particular, it is sufficient to consider $A$ of form $\{f\geq\alpha\}$ (or $\{f\leq \alpha\}$) for some observable $f:S^{\T_N^d}\to[0,1]$ and some threshold $\alpha\in(0,1)$. 

    To construct the appropriate event, we first consider a radius $r:=2e^2dD\log N$ and a collection $\CC\subset\T_N^d$ of $N^d/(2d)^{2r}$ ``centres'' at distance at least $2r$ apart. Writing $s_{\max}$ and $s_{\min}$ for the maximal and the minimal element of $S$ we consider two copies $(\sigma_t^+,\eta_t^+)_{t\geq 0}$ and $(\sigma_t^-,\eta_t^-)_{t\geq 0}$ of the $\Phi^{(N)}$-spin system on dynamical percolation, started from $\sigma_0^+\equiv s_{\max}$, $\sigma_0^-\equiv s_{\mix}$ and $\eta_0^+,\eta_0^-\sim\P_p$, which we couple so that $\eta_0^+=\eta_0^-$, $\upd(\sigma^+)=\upd(\sigma^-)$ and $\upd(\eta^+)=\upd(\eta^-)$. Moreover, we define a process $(\tilde\sigma_t,\tilde\eta_t)_{t\geq 0}$ by 
    \begin{itemize}
        \item[(i)] setting $\tilde\eta_0=\eta_0^+$ and sampling $\tilde\sigma_0\sim\pi(\cdot|\eta_0^+)$, so that the starting configuration is stationary,
        \item[(ii)] evolving $(\tilde\eta_t)_{t\geq 0}$ using $\upd(\eta^+)$, so that $\tilde\eta_t=\eta_t^+=\eta_t^-$, for all $t\geq 0$, and
        \item[(iii)] evolving $(\tilde\sigma_t)_{t\geq 0}$ using 
        $$\upd(\tilde\sigma)~:=~\{(V,T,U)\in\upd(\sigma^+):\dist(V,\CC)<r\},$$
        so that outside $\cup\{B_{r-1}(x):x\in\CC\}$, where $B_{r-1}(x)$ is the $(r-1)$-ball around $x$, spins are not updated.
    \end{itemize}
    Finally, letting $(S^-,S^+)$ be a partition of $S$ satisfying the non-triviality condition given in \eqref{eq:nonrtiv_condition}, define an observable 
    $$f^+(\sigma) ~:=~ \frac{1}{|\CC|}\sum_{x\in\CC}\1_{\{\sigma(x)\in S^+\}}, \quad \sigma\in S^{\T_N^d},$$
    the proportion of centres at which $\sigma$ assumes a value in $S^+$. Moreover, we write $f^-:=1-f^+$ for the proportion of centres at which $\sigma$ assumes a value in $S^-$. 

    \noindent We now distinguish between two possible cases:
    \begin{itemize}
        \item[(A)] $\pi(f^+\leq\E[f^+(\tilde\sigma_T)])\geq 1/2$,
        \item[(B)] $\pi(f^+>\E[f^+(\tilde\sigma_T)])\geq1/2$. 
    \end{itemize}

    \noindent\textbf{Case A:} In this case, it is sufficient to prove that for some $\varepsilon=\varepsilon_N>0$,
    \begin{equation}\label{eq:suff_case_A}
    \P\big(f^+(\sigma_T^+)<\E[f^+(\tilde\sigma_T)]+\varepsilon\big) ~<~ \frac{1}{2}-\kappa,
    \end{equation}
    for $N$ sufficiently large. In particular, it is sufficient to show that the value on the left hand side approaches $0$ as $N\to\infty$. As a tool, we define a process $(\bar\sigma_t,\bar\eta_t)_{t\geq 0}$ by
    \begin{itemize}
        \item[(i)] setting $\bar\sigma_0\equiv s_{\max}$ and $\bar\eta_0=\eta_0^+$,
        \item[(ii)] evolving $(\bar\eta_t)_{t\geq 0}$ using $\upd(\eta^+)$, and
        \item[(iii)] evolving $(\bar\sigma_t)_{t\geq 0}$ using $\upd(\tilde\sigma)$.
    \end{itemize}
    Note that by definition, $(\tilde\sigma_t,\tilde\eta_t)_{t\geq 0}$ and $(\bar\sigma_t,\bar\eta_t)_{t\geq 0}$ are coupled so that for each $t\geq 0$, $\tilde\eta_t=\bar\eta_t$ and $\tilde\sigma_t\leq\bar\sigma_t$. In order to show \eqref{eq:suff_case_A}, we exploit that (i) restrictions of $(\sigma_t^+)_{t\geq 0}$ and $(\bar\sigma_t)_{t\geq 0}$ to $(r-1)$-balls around centres typically do not manage to reach a disagreement until time T, and (ii) due to independent evolutions on distinct balls, $f^+(\bar\sigma_T)$ concentrates around its mean. In particular, we use that for 
    \begin{equation}\label{eq:appendix_epsilon_def}
    \varepsilon~=~\frac{1}{4}\pi[f^-]e^{-T},
    \end{equation}
    the left hand side in \eqref{eq:suff_case_A} is bounded from above by
    $$\P\big(f^+(\sigma_T^+)\leq f^+(\bar\sigma_T)-\varepsilon\big) + \P\big(f^+(\bar\sigma_T)\leq\E[f^+(\tilde\sigma_T)]+2\varepsilon\big).$$
    In order to give an upper bound on the probability of $\{f^+(\sigma_T^+)\leq f^+(\bar\sigma_T)-\varepsilon\}$, we note that this event implies that there is at least $|\CC|\varepsilon$ centres at which $\sigma_T^+$ and $\bar\sigma_T$ disagree, so the Markov's inequality yields
    $$\P\big(f^+(\sigma_T^+)\leq f^+(\bar\sigma_T)-\varepsilon\big) ~\leq~ \frac{1}{|\CC|\varepsilon}\sum_{x\in\CC}\P(\sigma_T^+(x)\neq\bar\sigma_T(x)).$$
    Fixing $x\in\CC$, the initial agreement $\sigma_0^+=\bar\sigma_0$ implies that the event $\{\sigma_T^+(x)\neq\bar\sigma_T(x)\}$ is contained in the event that there exists an $\ell^1$-path $x_1,\ldots,x_m$ where $d(x_1,x)\geq r$ and $x_m=x$, as well as a sequence of times $t_1<\ldots<t_m$ such that the Poisson clock associated with site $x_i$ rings at time $t_i$. By Observation 3.2 in \cite{Hayes_Sinclair_2007} and the fact that $m\geq r$, the probability of such event is bounded from above by $(eT2d/r)^r$ and hence
    $$\P\big(f^+(\sigma_T^+)\leq f^+(\bar\sigma_T)-\varepsilon\big) ~<~ \frac{e^{-r}}{\varepsilon}.$$
    In order to obtain an upper bound on the probability of $\{f^+(\bar\sigma_T)\leq\E[f^+(\tilde\sigma_T)]-2\varepsilon\}$, it is sufficient for us to give an upper bound on $\E[f^+(\tilde\sigma_T)]$ of form $\E[f^+(\bar\sigma_T)]-\varepsilon'$ for some $\varepsilon'>2\varepsilon$, as this allows us to make use of Hoeffding's bound. Since $\tilde\sigma_t\leq\bar\sigma_t$ for all $t\geq 0$, the event $\{\bar\sigma_T(x)\in S^+\}$ can be written as
    $$\{\tilde\sigma_T(x)\in S^+\}\cup\{\tilde\sigma_T(x)\in S^-,\bar\sigma_T(x)\in S^+\},$$
    so in particular
    $$\E[f^+(\bar\sigma_T)] ~=~ \E[f^+(\tilde\sigma_T)] + \frac{1}{|\CC|}\sum_{x\in\CC}\P(\tilde\sigma_T(x)\in S^-,\bar\sigma_T(x)\in S^+).$$
    We also exploit that under this coupling, the event $\{\tilde\sigma_T(x)\in S^-,\bar\sigma_T(x)\in S^+\}$ contains the event that the clock associated with site $x$ has not rung until time $T$ and that $\tilde\sigma_0(x)\in S^-$. Since clock rings are independent of the starting configuration, this event has probability precisely $\P(\tilde\sigma_0(x)\in S^-)e^{-T}$ and hence
    $$\E[f^+(\bar\sigma_T)] ~\geq~ \E[f^+(\tilde\sigma_T)] + \pi[f^-]e^{-T} ~=~ \E[f^+(\tilde\sigma_T)]+4\varepsilon.$$
    In particular,
    \begin{align*}
     \P\big(f^+(\bar\sigma_T)\leq\E[f^+(\tilde\sigma_T)]+2\varepsilon\big) ~&\leq~ \P\big(f^+(\bar\sigma_T)\leq\E[f^+(\bar\sigma_T)]-2\varepsilon\big) \\
     &\leq~ \exp\!\Big(\!-\frac{|\CC|}{2}\pi[f^-]^2e^{-2T}\Big),
    \end{align*}
    applying Hoeffding's bound in the second inequality. To conclude the proof of this case, we first note that
    $$\frac{e^{-r}}{\varepsilon} ~=~ \frac{4}{\pi[f^-]}e^{D(1-2e^2d)\log N},$$
    which approaches $0$ as $N\to\infty$; moreover, by definition of $|\CC|$, we obtained
    \begin{align*}
        |\CC|\pi[f^-]^2e^{-2T} ~&=~ \pi[f^-]^2e^{d\log N-2D\log N(2e^2d\log(2d)-1)} \\
        &\geq~ \pi[f^-]^2e^{(1-4e^2D\log(2d))d\log N},
    \end{align*}
    where the right hand side goes to infinity by choosing $D<(4e^2\log(2d))^{-1}$. Note that here we are using that $\pi[f^-]\geq\chi$, uniformly in $N$, where $\chi>0$ is the constant given in \eqref{eq:nonrtiv_condition}. Overall, we obtained that by letting $\varepsilon$ be as in \eqref{eq:appendix_epsilon_def} and $T=D\log N$ with $D<(4e^2\log(2d))^{-1}$,
    $$\mathbb{P}\big(f^+(\sigma_T^+)<\mathbb{E}[f^+(\tilde\sigma_T)]+\varepsilon\big) ~<~ \frac{e^{-r}}{\varepsilon}+\exp\!\Big(\!-\frac{|\CC|}{2}\pi[f^-]^2e^{-2T}\Big),$$
    where the right hand side converges to $0$ as $N\to\infty$, so in particular \eqref{eq:suff_case_A} holds for all $N$ large enough, concluding the proof of this case. \\

    \noindent\textbf{Case B:} The proof in this case is largely analogue to the one in case A. In this case, it is sufficient to show that for some $\varepsilon=\varepsilon_N$, 
    \begin{equation}\label{eq:suff_case_B}
        \P\big(f^+(\sigma_T^-)>\E[f^+(\tilde\sigma_T)]-\varepsilon\big) ~\to~ 0,\quad \text{as}~N\to\infty.
    \end{equation}
    Similarly as before, we define a process $(\hat\sigma_t,\hat\eta_t)_{t\geq 0}$ by
    \begin{itemize}
        \item[(i)] setting $\hat\sigma_0\equiv s_{\min}$ and $\hat\eta_0=\eta_0^-$,
        \item[(ii)] evolving $(\hat\eta_t)_{t\geq 0}$ using $\upd(\eta^-)$, and
        \item[(iii)] evolving $(\hat\sigma_t)_{t\geq 0}$ using $\upd(\tilde\sigma)$. 
    \end{itemize}
    By definition, $(\tilde\sigma_t,\tilde\eta_t)_{t\geq 0}$ and $(\hat\sigma_t,\hat\eta_t)_{t\geq 0}$ are coupled so that for each $t\geq 0$, $\tilde\eta_t=\hat\eta_t$ and $\tilde\sigma_t\geq\hat\sigma_t$. We now use that for $\varepsilon=4^{-1}\pi[f^+]e^{-T}$, the left hand side in \eqref{eq:suff_case_B} is bounded from above by
    $$\P\big(f^+(\sigma_T^-)\geq f^+(\hat\sigma_T)+\varepsilon\big) + \P\big(f^+(\hat\sigma_T)\geq\E[f^+(\tilde\sigma_T)]-2\varepsilon\big).$$
    An identical disagreement percolation argument as above yields that
    $$\P\big(f^+(\sigma_T^-)\geq f^+(\hat\sigma_T)+\varepsilon\big) ~\leq~ \frac{e^{-r}}{\varepsilon}.$$
     We are now left with bounding appropriately the probability of $\{f^+(\hat\sigma_T)]\geq\E[f^+(\tilde\sigma_T)]-2\varepsilon\}$, which is done in a very similar matter as in case A. In this particular case, it is sufficient to give a lower bound on $\E[f^+(\tilde\sigma_T)]$ of form $\E[f^+(\hat\sigma_T)]+\varepsilon'$ for some $\varepsilon'>2\varepsilon$, which we do exploiting that $\hat\sigma_t\leq\tilde\sigma_t$ for all $t\geq 0$. The latter tells us that 
    $$\{\tilde\sigma_T(x)\in S^+\} ~=~ \{\hat\sigma_T(x)\in S^+\}\cup\{\hat\sigma_T(x)\in S^-,\tilde\sigma_T(x)\in S^+\},$$
    and hence 
    $$\E[f^+(\tilde\sigma_T)] ~=~ \E[f^+(\hat\sigma_T)] + \frac{1}{|\CC|}\sum_{x\in\CC}\P(\hat\sigma_T(x)\in S^-,\tilde\sigma_T(x)\in S^+).$$
    By definition of the coupling, the event $\{\hat\sigma_T(x)\in S^-,\tilde\sigma_T(x)\in S^+\}$ contains the event that the clock associated with site $x$ has not rung until time $T$ and that $\tilde\sigma_0(x)\in S^+$, which has the probability $\P(\tilde\sigma_0(x)\in S^+)e^{-T}$ and hence
    $$\E[f^+(\tilde\sigma_T)] ~\geq~ \E[f^+(\hat\sigma_T)] + \pi[f^+]e^{-T} ~=~ \E[f^+(\hat\sigma_T)]+4\varepsilon.$$
    In particular,
    \begin{align*}
        \P\big(f^+(\hat\sigma_T)\geq\E[f^+(\tilde\sigma_T)]-2\varepsilon\big) ~&\leq~ \P\big(f^+(\hat\sigma_T)\geq\E[f^+(\hat\sigma_T)]+2\varepsilon\big) \\
        &\leq~ \exp\!\Big(\!-\frac{1}{2}|\CC|\pi[f^+]^2e^{-2T}\Big).
    \end{align*}
    Using that, by assumption \eqref{eq:nonrtiv_condition}, $\pi[f^+]\geq\chi$ uniformly in $N$, we can recycle from above that $e^{-r}/\varepsilon\to 0$ and $|\CC|\pi[f^+]^2e^{-2T}\to\infty$ as $N\to\infty$ in order to obtain that \eqref{eq:suff_case_B} indeed holds, which concludes the proof.
\end{proof}


\nocite{*}
\bibliographystyle{plain}  
\bibliography{references}  


\appendix

\section{Validity of coupling $\Q_f$}\label{app:coupling}

Here we verify that $\Q_{f,g}$ in fact defines a coupling of $\Prm_f$ and $\Prm_g$, in the sense that for each $t\geq 0$, the marginals of the law of $(\sigma_t,\eta_t)$ and $(\sigma_t',\eta_t')$ under $\Q_{f,g}$ correspond to $\Prm_f^t$ and $\Prm_g^t$ respectively. 

It is immediate, due to the implementation of the identity coupling that both $(\eta_t)_{t\geq 0}$ and $(\eta_t')_{t\geq 0}$ have law $\Psf:=\sum_{\eta}\P_p(\eta)\Psf_\eta(\cdot)$. In particular, for any $t\geq 0$, the marginal law of $\eta_t$ (resp.~$\eta_t'$) under $\Q_{f,g}$ is $\P_p$. 

We now write $\I:=[0,\varepsilon/\lambda]$ and recall that for fixed $k\geq 0$, we defined $t_k=k\varepsilon/\lambda$. For $t,t'\geq 0$, we write $\F_{t,t'}$ for the $\sigma$-algebra generated by $(V,T,U)\in\upd(\sigma)\cup\upd(\sigma')$ such that $T\leq t$ and $(E,T,U)\in\upd(\eta)\cup\upd(\eta')$ such that $T\leq t'$, as well as the realization of the (random) initial configurations. Additionally, we write $\F_t:=\F_{t,t}$. It is sufficient to verify that for any given $\xi\in S^{\T_N^d}$, $k\geq 0$ and $s\in\I$,
\begin{align}
    \Q_{f,g}(\sigma_{t_k+s}=\xi|\F_{t_k}) ~&=~ \Prm_{\sigma_{t_k},\eta_{t_k}}^s(\xi), \label{eq:appendix_to_verify_1} \\
    \Q_{f,g}(\sigma_{t_k+s}'=\xi|\F_{t_k}) ~&=~ \Prm_{\sigma_{t_k}',\eta_{t_k}'}^s(\xi). \label{eq:appendix_to_verify_2}
\end{align}
By definition of the coupling, i.e., since $(\sigma_{t},\eta_t)_{t\geq 0}$ is evolved using $\upd(\sigma)$ and $\upd(\eta)$, it is immediate that \eqref{eq:appendix_to_verify_1} holds; it remains to verify \eqref{eq:appendix_to_verify_2}. Writing $\EE_{f,g}$ for the expectation under $\Q_{f,g}$, it follows from the tower property that 
$$\Q_{f,g}(\sigma_{t_k+s}'=\xi|\F_{t_k}) ~=~ \EE_{f,g}\big[\Q_{f,g}(\sigma_{t_k+s}'=\xi|\F_{t_k,t_{k+1}})\big|\F_{t_k}\big];$$
it is thus sufficient to verify that
$$\Q_{f,g}(\sigma_{t_k+s}'=\xi|\F_{t_k,t_{k+1}}) ~=~ \Prm_{\sigma_{t_k}',\eta_{t_k}'}^s\big(\xi\big|(\eta_t')_{t\in\I}\big).$$
Recalling the notion of ajar clusters from Section \ref{subsec:boxes}, we slightly abuse the notation to write $\C_{k+1}^\ajar$ for the collection of ajar clusters associated with the interval $\I_{k+1}=[t_k,t_{k+1})$. We note that $\C_{k+1}^\ajar$ forms a partition of $\T_N^d$ and recall that for any distinct $\C,\C'\in\C_{k+1}^\ajar$,
$$\big(\sigma_{t}'(\C):t\in\I_{k+1}\big)\quad\text{and}\quad\big(\sigma_{t}'(\C'):t\in\I_{k+1}\big)$$
are independent. Thus, since the collection $\C_{k+1}^\ajar$ is measurable with respect to $\F_{t_k,t_{k+1}}$, we obtain the following decomposition:
\begin{equation}\label{eq:ajar_decomp}
    \Q_{f,g}(\sigma_{t_k+s}'=\xi|\F_{t_k,t_{k+1}}) ~=~ \prod_{\C\in\C_{k+1}^\ajar}\Q_{f,g}\big(\sigma_{t_k+s}'(\C)=\xi(\C)\big|\F_{t_k,t_{k+1}}\big).
\end{equation}
Fixing some $\C\in\C_{k+1}^\ajar$, we distinguish between two scenarios:
\begin{itemize}
    \item[(i)] If $\sigma_{t_k}(\C)\neq\sigma_{t_k}'(\C)$ and no edge updates occurred to the edges in the $\C$, i.e., there exist no $(E,T,U)\in\upd(\eta)$ with $E\in E^+(\C)$ and $t_k\leq T<t_{k+1}$, then the restriction of $\Q_{f,g}$ to $\C$ and $\I_{k+1}$ corresponds to the $(\varepsilon/\lambda)$-optimal coupling on $\C$ (described in Section \ref{sec:prelim}), keeping in mind that the absence of edge updates implies that $\C=\C_{s}(x)$ for arbitrary $s\in\I_{k+1}$ and $x\in V(\C)$. In particular,
    \begin{align*}
    \Q_{f,g}\big(\sigma_{t_k+s}'(\C)=\xi(\C)\big|\F_{t_k,t_{k+1}}\big)  ~&=~ \Prm_{\sigma_{t_k}',\eta_{t_k}'}\big(\sigma_s'(\C)=\xi(\C)\big|(\eta_t')_{t\in\I}\big),
    \end{align*}
    noting that $(\sigma_{t_k}',\eta_{t_k}')$ is understood as a realization of random variables, measurable with respect to $\F_{t_k,t_{k+1}}$.
    \item[(ii)] In the opposite case, the restriction of $\Q_{f,g}$ to $\C$ and $\I_{k+1}$ corresponds to the identity coupling and hence
    \begin{align*}
    \Q_{f,g}\big(\sigma_{t_k+s}'(\C)=\xi(\C)\big|\F_{t_k,t_{k+1}}\big) ~&=~ \Prm_{\sigma_{t_k}',\eta_{t_k}'}\big(\sigma_s'(\C)=\xi(\C)\big|(\eta_t')_{t\in\I}\big).
    \end{align*}
\end{itemize}
It thus follows that
\begin{align*}
    \Q_{f,g}(\sigma_{t_k+s}'=\xi|\F_{t_k,t_{k+1}}) ~&=~ \prod_{\C\in\C_{k+1}^\ajar} \Prm_{\sigma_{t_k}',\eta_{t_k}'}\big(\sigma_s'(\C)=\xi(\C)\big|(\eta_t')_{t\in\I}\big) \\
    &=~  \Prm_{\sigma_{t_k}',\eta_{t_k}'}^s\big(\xi\big|(\eta_t')_{t\in\I}\big),
\end{align*}
concluding the proof.


\end{document}